\documentclass[12pt]{article}
\usepackage{amsfonts,amsmath,latexsym,amssymb,euscript,graphicx,units,mathrsfs, 
  amsthm}

\usepackage{caption,subcaption}

\usepackage{hyperref}

\newcommand{\beas}{\begin{eqnarray*}}
\newcommand{\enas}{\end{eqnarray*}}
\newcommand{\bea}{\begin{eqnarray}}
\newcommand{\ena}{\end{eqnarray}}

\newcommand{\FFF}{\mathcal{F}}
\newcommand{\HHH}{\mathcal{H}}

  \newcommand{\AAA}{\mathcal{A}}

\newcommand{\EE}{\mathbb{E}}
  \newcommand{\RR}{\mathbb{R}}
  \def\vvvert{|\!|\!|}

\usepackage{graphicx}
\usepackage{color}
\usepackage{amssymb}
\usepackage[T1]{fontenc}
\usepackage{latexsym}
\usepackage{xypic}
\usepackage{eufrak}
\usepackage{euscript}
\usepackage{amsfonts,amsmath}
\usepackage{verbatim}
\usepackage{fancyhdr}
\usepackage[english]{babel}
\usepackage{mathrsfs}
\usepackage{units}

\newtheorem{prop}{Proposition}[section]
\newtheorem{cor}[prop]{Corollary}
\newtheorem{lemma}[prop]{Lemma}
\newtheorem{remark}[prop]{Remark}
\newtheorem{rem}[prop]{Remark}

\newtheorem{thm}[prop]{Theorem}

\newtheorem{theorem}[prop]{Theorem}
\newtheorem{definition}[prop]{Definition}
\newtheorem{example}[prop]{Example}

\newcommand{\A}{\mathcal A}

\renewcommand{\geq}{\geqslant}
\def\leq{\leqslant}

\newcommand{\N}{\mathbb{N}}

\newcommand{\R}{\mathbb{R}}

\newcommand{\E}{\mathbb{E}}
\renewcommand{\P}{\mathbb{P}}
\newcommand{\eq}{\begin{equation}}
\newcommand{\qe}{\end{equation}}
\newcommand{\dv}{\mathrm{div}}

\newcommand{\grad}{\nabla}

\usepackage{color}

\begin{document}

\title{Stein operators, kernels and discrepancies for multivariate
  continuous distributions} \author{G. Mijoule\footnote{INRIA Paris,
    Team MoKaPlan, 2 rue Simone Iff, 75012 Paris,
    guillaume.mijoule@inria.fr}, \; G. Reinert\footnote{University of
    Oxford, Department of Statistics, 1 South Parks Road, Oxford OX1
    3TG, UK, reinert@stats.ox.ac.uk} \; and
  Y. Swan\footnote{Universit\'e libre de Bruxelles, D\'epartement de
    Math\'ematique, Campus Plaine, Boulevard du Triomphe CP210, B-1050
    Brussels, yvswan@ulb.ac.be} }

\date{}
\maketitle

\abstract{We present a general framework for setting up
  Stein's method for multivariate continuous distributions. The
  approach gives a collection of Stein characterizations, among which
  we highlight score-Stein operators and kernel-Stein
  operators. Applications include copulas and distance between
  posterior distributions. We give a general explicit construction for
  Stein kernels for elliptical distributions and discuss Stein kernels
  in generality, highlighting connections with Fisher information and
  mass transport. Finally, a goodness-of-fit test based on Stein
  discrepancies is given.}

\medskip


\section{Introduction} 
\label{sec:introduction}

 Stein's method is a
collection of tools permitting to bound quantities of the form 
\begin{equation*}
d_{\mathcal{H}}(X, W) =  \sup_{h \in \mathcal{H}} \left| \mathbb{E}[h(W)] - \mathbb{E}[h(X)] \right|
\end{equation*}
where $\HHH$ is a measure-determining class and $X, W$ are two random
quantities of interest with $X$, say, {following the target distribution $\mu$}.
The method can be
summarised as follows.  First 
find a so-called {{\it Stein operator}} ${\AAA}$ and a wide class
of functions $\FFF(\AAA)$ such that (i)~$X \sim \mu$ if and only if
$ \E [ \AAA f(X)] = 0$ for all functions $f \in \FFF(\AAA)$ \emph{and}
(ii)~for each $h \in \HHH$ there exists a well-defined and tractable
solution $f = f_h \in \FFF(\AAA)$ of the {{\it Stein equation}}
$h(x) - \E [h(X)] = \AAA f(x).$ Then, upon noting that
$ \E [h(W)] - \E [h(X)] = \E [\AAA f_h (W)]$ for all $h$,
the problem of bounding $d_{\mathcal{H}}(X, W)$ has been re-expressed
in terms of that of bounding
$ \sup_{h \in \mathcal{H}} \E [\AAA f_h (W)]$. The success of the
method lies in the fact that this last quantity is amenable to a wide
variety of approaches.

{There exist many frameworks in which Stein's method is well
  understood. We refer to the surveys \cite{rossfund,BaCh14,chatsurv}
  as well as the papers \cite{LRS16,ley2015distances} for the
  univariate setting.  We also highlight \cite{ahind2019} who provide
  an up-to-date overview in the context of infinitely divisible
  distributions. Comprehensive introductions to some of the most
  important aspects of the theory are available from the monographs
  \cite{stein1986} as well as \cite{NP11,ChGoSh11}, with a particular
  focus on Gaussian approximation. Although the univariate case is the
  most studied, many references also tackle multivariate
  distributions. For discrete multivariate distributions,
  \cite{barbour2015multivariate} and \cite{barbour2018multivariate}
  provide a framework which is applicable in many situations. In
  \cite{reinert2017approximating}, stationary distributions of Glauber
  Markov chains are characterised. The multivariate Gaussian case has
  naturally received the most attention, starting with \cite{Ba90} and
  \cite{G91}, see also \cite{MR2573554,
    FSW18} 
  and \cite{GMS18}.
The Dirichlet distribution has been
treated in \cite{gan2017dirichlet} using a coupling
approach. Log-concave densities are treated in
\cite{mackey2016multivariate}, and more general settings have been
dealt with in \cite{gorham2016measuring,FSW18}. 
Yet,  a general approach for the
multivariate case 
  has been elusive. This is the task which our paper
  addresses.}

  The starting point for a
   multivariate Stein's method is a Stein characterization for the
   Gaussian law which states that a random vector $Y$ is a
   multivariate Gaussian $d$-random vector with mean $\nu$ and
   covariance $\Sigma$ (short: $Y \sim {\cal{N}}(\nu, \Sigma)$) if and
   only if
\begin{equation}
  \label{eq:stlemma}
    \EE \big[(Y-\nu)^t \nabla f(Y)\big]  
    = \EE\big[ \nabla^t \Sigma \nabla f(Y)\big], 
\end{equation}
for all absolutely continuous function $f : \R^d \to \R$ for which the
expectations exist; here $\nabla$ denotes the gradient
operator. Assume that $h:\R^d\to \R$ has three bounded derivatives. Then,
if $\Sigma\in\R^{d\times d}$ is symmetric and positive definite, and
$Z \sim \mathcal{MVN}(0,\Sigma)$, there is a solution $f:\R^d\to \R$
to the Stein equation
\begin{equation}\label{mvnstein}
    \nabla^t \Sigma \nabla f(w) - w^t\nabla f(w)
     = h(w) - \EE h(\Sigma^{1/2}Z),
\end{equation}
given at every $w\in\R^d$ by the {\it{Mehler formula}}
\begin{equation}\label{mehler}f(w) = \int_0^1 \frac{1}{2t} \big(\EE
  {h}( Z_{w,t} ) - \EE {h}(\Sigma^{1/2}Z)\big)
  dt\end{equation} 
with $Z_{w,t} = \sqrt{t} w+ \sqrt{1-t} \, \Sigma^{1/2}Z. $ 
Generalizations
to infinitely-dimensional functionals using Malliavin calculus are
available in \cite{NP11} and a generalization to Gaussian random
measures is employed in \cite{holmes2004stein}. 

{Pursuing the analogy,
  classical Markov theory leads, under appropriate regularity
  assumptions,  to Stein equations of the form 
$
 \left\langle  a(x), \nabla^2 f(x) \right\rangle + \left\langle b(x), \nabla f(x) \right\rangle = h(x) - \EE[h(Z)]
$
for targets $Z \sim \mu$ the ergodic measure of SDE's of the form
$\mathrm{d}Z_t = a(t) \mathrm{d} B_t + b(X_t) \mathrm{d}t$; see e.g.\
\cite{BGL13,gorham2016measuring,FSW18}.  As in the Gaussian case, a solution $f_h$ to the Stein
equation is also provided by semigroup approach to Markov generators. 



 In this paper we shall use first order directional derivatives as basic building blocks of Stein operators. 
In Definition \ref{def:steclass} we define the \emph{canonical Stein derivative for $p$ in the direction
      $e$} as  
      $  \mathcal{T}_{e, p} \phi =  {\partial_e( p \, 
   \phi)} / {p}
$
acting on sufficiently smooth test functions $\phi $. 
  Using this building block, we define the \emph{Stein gradient operator}
   acting on real valued functions $f$, vector valued functions
   $\mathbf{f}$ or  matrix valued functions $\mathbf{F}$ as
   $\bullet \mapsto \mathcal{T}_{\nabla, p}  \bullet ={\nabla \left( p \bullet
     \right)} \ {p}= \sum_{i=1}^d e_i\mathcal{T}_{e_i, p}$
     where $\{ e_i, i =1, \ldots, d\}$ is the canonical basis of $\R^d$. 
Similarly we define the \emph{Stein divergence operator} acting on
   vector or matrix valued functions with compatible dimensions as  
  $\bullet \mapsto \mathcal{T}_{\mathrm{div}, p}  \bullet = {\mathrm{div}\left( p \bullet
    \right)}/ {p}.$ 
    As in the one-dimensional case, this operator satisfies a product rule  for all  smooth functions $f, g$, namely 
$
  \mathcal{T}_{e, p}(f\, g) =   (\mathcal{T}_{e, p}f) g +
  f (\partial_e g ) 
$
for all $e $ in the $d$-dimensional unit sphere $S^{d-1}$. 

Our framework
provides a mechanism for obtaining a broad
family
of Stein operators which we call \emph{standardizations} of the
canonical operator $\mathcal{T}_{\mathcal{D}, p}$, where
$\mathcal{D} = \nabla$ or ${\rm{div}}$ -- see Section
\ref{sec:first-defin-outl}.  We do not claim that all useful operators
from the literature are of this form -- for instance
\cite{bonis2015rates} proposes an alternative construction. We do
demonstrate, however, that the operators we construct provide crucial
handles on a wide family of multivariable distributions.

For example, setting $\rho = \grad \log p = \mathcal{T}_{\nabla, p}1$
the multivariate score function, the analog to the one-dimensional
score-Stein operator is the vector valued operator
$      \mathcal{A}_pg = \nabla g  +   \rho \, g$
    acting on differentiable functions $g : \R^d \to \R$, see
    Definition \ref{def:multiscore}. Many authors call this
    score-Stein operator \emph{the} Stein operator for multivariate
    $p$; see e.g.\ \cite[definition 2.1]{liu2016kernelized},
    \cite[Section 2.1]{chwialkowski2016kernel}, \cite[Section
    2.2.2]{oates2017control}. 
    
    A second set of Stein operators are related to Stein kernels.  In
    dimension $d=1$, 
    this is
   the unique bounded solution $x \mapsto {\tau_p}(x)$  in
    $\mathcal{F}(p)$ to the ODE
    $ \mathcal{T}_p({\tau_p}(x)) =\nu-x,$ given by
    \begin{equation} \label{steinkernel1d}
{{\tau_p}}(x) = \frac{1}{p(x)} \int_x^{\infty} (y-\nu) p(y) dy, \quad x \in \mathcal{I}
\end{equation}
({with} $\nu = \E_p[X]$).  Properties of the Stein kernel were
first studied in \cite{stein1986} (although it had long been a known
important handle on smooth densities $p$, see e.g.\ \cite{CP89,C82},
who refer to it as a covariance kernel).  The Stein kernel has become
a powerful tool in Stein's method, see e.g.~\cite{CPU94,NoVi09, NP11}
or the more recent works \cite{F2018,courtade2017existence}.  In one
dimension, Stein kernels are unique when they exist - the Stein kernel
is the zero bias density from \cite{goldstein1997stein}. They have
been studied in detail in \cite{ernst2019first,saumard2018weighted}.  In higher
dimensions even the definition of a Stein kernel is not obvious; see
\cite{nourdin2013entropy,nourdin2013integration,courtade2017existence,
  ledoux2015stein}.  The zero-bias coupling definition in
\cite{goldstein2005zero} uses a collection of random elements in
higher dimensions. In analogy, in Definition \ref{def:multiteikdef} we
define a {\it{directional Stein kernel}} for each canonical direction
$e_i \in \R^d$, as any function
$x \mapsto \tau_{p,i}(x) \in \R\times \R^d$ belonging to
$\mathcal{F}(p)$ such that
$ \mathcal{T}_{\mathrm{div}, p}\big(\tau_{p, i}(x)\big) =\E_p[X_i]-
x_i$ for all $x \in \mbox{supp}(p). $ A \emph{Stein kernel} is then
\emph{any} square matrix $ \pmb\tau$ such that each line
$ \pmb \tau_i = (\tau_{i1}, \ldots, \tau_{id})$ is a kernel in the
direction $e_i$, and the \emph{kernel-Stein operator} is the operator
$
  \mathcal{A}_pg(x) =  \pmb\tau_p(x) \nabla g(x)- (x-\nu)g(x).
$
Stein kernels need not exist and may not be unique when they exist.
Aside from discussing this issue in Section
\ref{sec:stein-discrepancies-1}, we also give {an}
explicit
 --
formula for Stein kernels of elliptical distributions in Section
\ref{sec:ellipt-distr}. 
 In particular, we
obtain new identities even for the standard Gaussian: aside from
Stein's classical Gaussian covariance identity
\begin{equation}\label{steincov}
X \sim \mathcal{N}(\nu, \Sigma)  \Rightarrow   \E [\Sigma \nabla g(X)
] = \E [(X-\nu) g(X)] 
\end{equation}
we also prove that 
for
all $\beta \neq 2$ and $d \neq 1$, if
$X \sim \mathcal{N}(\nu, \Sigma) $ then
\begin{equation*}
\frac{\beta}{(\beta-2)(d-1)}  \EE \left[\big(X^T\Sigma^{-1}X\Sigma -X
  X^T \big) \nabla g(X) \right]  = \EE \left[ \frac{2}{\beta-2}
  \Sigma \nabla  g(X)  + X g(X)\right];
\end{equation*} 
both identities hold for a wide family of test functions $g$, see
Section \ref{sec:stein-oper-gauss}. Also, in the case of the
multivariate Student $t$-distribution, we obtain the identity
\begin{equation*}
X \sim t_k(\nu, \Sigma) \Rightarrow    \E  \left[ (X X^T+k I_d)\nabla
  g(X)  \right] =  (k-1)\E  \left[ X
    g(X)\right]
\end{equation*}
which 
generalizes 
 the corresponding
univariate result (see Section \ref{sec:mult-stud-t}).  We also
provide, in Section \ref{sec:stein-discrepancies-1}, a connection
between Stein kernels and transport maps. These results are closely
linked to the seminal works \cite{artstein2004solution,
  BaBaNa03,ABBNptrf}; we show that the key quantity they introduce
to solve e.g.\ Shannon's conjecture actually is a Stein kernel. In
particular, these results give new and explicit expressions for Stein
kernels in $d=2$ and $d=3$ dimensions.

A natural question is which Stein operator to choose. Here practical
issues may serve as guidance: solutions to multivariate Stein
equations may not generally exist. In the case of densities which
possess a Poincar\'e constant, weak solutions of second-order Stein
equations exist, and some regularity results for these solutions are
available, {see} 
\cite{MR2091552,
  MR2573554,ChMe08,valiant2010clt, gorham2016measuring,FSW18,GMS18}, see 
  Section
\ref{sec:bound-stein-discr}.
   
We {illustrate} our framework 
{by assessing} the (1-)Wasserstein distance between two absolutely
continuous distributions. This result leads to {three
  applications}: assessing the difference between copulas,
assessing
the distance between prior and posterior in a Bayesian setting,
{and bounding the}
Wasserstein distance between skew-normal and standard normal
distributions.

In practical applications and for obtaining information inequalities,
the {\it{Stein discrepancy}} plays a key role. The Stein discrepancy
from distribution $p$ to distribution $q$ using the Stein operator
$\AAA_p$ for $p$ is given in \eqref{eq:stediscrepancy};
$
  \mathcal{S}_{\left\| \cdot \right\|}(q, \mathcal{A}_p, \mathcal{G}) = \sup_{g \in
    \mathcal{G}} \left\| \E \left[ \mathcal{A}_pg(Y) \right] \right\|.
$
Here $Y \sim q$ and $\mathcal{G}$ is a class of smooth test
functions. A particular case is that of kernelized Stein
discrepancies, introduced in
\cite{liu2016kernelized,chwialkowski2016kernel}. We assess the
performance of such a kernelized Stein discrepancy measure for
simulations from a centred Student $t$-distribution, employing a Stein
discrepancy for a goodness-of-fit test.


The paper is structured as follows. Section \ref{sec:notations-1}
introduces the notations and the product rule.
 Section \ref{sec:general-theory} gives the
general Stein operators, Stein identities, and Stein
characterizations. Score-Stein operators and kernel-Stein operators
are introduced, and properties of the solution of Stein equations are
given for densities which admit a Poincar\'e constant. Applications of
this framework are given in Section \ref{sec:applications}. Then the
paper turns its attention to Stein kernels. In Section
\ref{sec:ellipt-distr}, Stein kernels for elliptical distributions are
derived. Section \ref{sec:stein-kernel} provides the general
discussion on Stein kernels. Finally, Section
\ref{sec:stein-discrepancies} discusses information metrics,
kernelized Stein discrepancies, and ends with a goodness-of-fit test
based on Stein discrepancies.

\section{Notations, gradients and product rules}
\label{sec:notations-1}

In this paper we use small caps for elements of
{$\R^m, m\ge 1$ (by convention column vectors),}
  capitalized letters for matrices in $\R^m\times \R^n$,
{small caps for real-valued functions,} boldfaced
  small caps for vector valued functions and boldfaced capitalized
  letters for matrix valued functions.

Fix $d \in \N_0$ and let $e_1, \ldots, e_d$ be the canonical basis for
Cartesian coordinates in $\R^d$ .
Given $x, y \in \R^d$ and any symmetric
positive-definite $d\times d$ matrix ${A}$ we set
$ \left\langle x, y \right\rangle_{A} = x^T {A} y$ (here $\cdot^T$
denotes the 
 transpose)
  with associated norm
$\|x\|_{A} = \sqrt{\left\langle x, x \right\rangle_{A}}$. 
With $\mathrm{Tr}(\cdot)$  the trace
operator, 
 the {\it{Hilbert-Schmidt scalar product}}  between  matrices ${A}, {B}$ of compatible
  dimension is
$\left\langle {A}, {B} \right\rangle_{\mathrm{HS}} = \mathrm{Tr}({A}
{B}^T)$, with associated norm
$\|{A}\|_{HS}^2 := {\mbox{Tr} ({A}^T{A})} = {\sum_{i,j=1}^d
  a_{i,j}^2}$.

Let $S^{d-1}$ denote the unit sphere in $\R^d$ and let $e \in S^{d-1}$
be a unit vector in $\R^d$. The directional derivative of a smooth
function $v:\R^d \to\R$ in the direction $e$ is the function
    $\partial_e v: \R^d \to \R$ given by
\begin{equation}
  \label{eq:45}
  \partial_ev(x) = \lim_{h \to 0} \frac{v(x+he)-v(x)}{h},
\end{equation}
at every point where  this limit exists. For $ i=1, \ldots, d$ we write $\partial_iv$ for the
derivative in the direction of the unit vector $e_i$.   The
\emph{gradient} of a smooth function $v :\R^d \to\R$ is
\begin{equation}
  \label{eq:7}
  \nabla v = {\mathrm{grad}(v)=} \left( {\partial_1 v}, \ldots,
  {\partial_d v} \right)^T = \sum_{i=1}^d \left({\partial_i v}\right)
e_i
\end{equation}
(by convention, a column vector). The \emph{gradient} of a
$1 \times m$ vector field $\mathbf{v}=(v_1, \ldots, v_m)$ {(a line
vector)} is
\begin{equation}
  \label{eq:40}
  \nabla \mathbf{v}=
  \begin{pmatrix}
    \nabla v_1 &  \cdots & \nabla v_m
  \end{pmatrix} =
  \begin{pmatrix}
     \partial_1 v_1 & \cdots & \partial_1 v_m \\
         \vdots    & \ddots  & \vdots \\
    \partial_d v_1 & \cdots & \partial_d v_m 
  \end{pmatrix}.
\end{equation}
(by convention, a $d\times m$ matrix). 

%
The \emph{Jacobian matrix} of a
$\R^m$-valued function
$\mathbf{w} = \left( w_1, \ldots, w_m \right)^T$ (a column vector) is
the $m \times d$ matrix
 $  \mathrm{Jac}(\mathbf{w})=
  \big(  \partial_i w_j
  \big)_{1\le i \le m, 1\le j \le d}.$
{As} $\mathrm{Jac}(\mathbf{v}^T)= \nabla \mathbf{v}$ 
  we will simply use the
  generic notation $\nabla$ for both operations from here onwards. The \emph{divergence} operator is defined for $d$ valued (line or
column) vector fields $\mathbf{v}$ with components
$v_j, j = 1, \ldots, d$ as
$$\mathrm{div}(\mathbf{v})
= \mathrm{Tr} \left(\nabla\mathbf{v} \right) =
\sum_{i=1}^d \partial_i v_i.$$

{
More generally, given any $m, n \in \N$, the directional derivative of
a matrix valued function
      \begin{eqnarray} \label{fform} \mathbf{F}: \R^d \rightarrow \R^{m}\times \R^d : \mathbf{x} \mapsto
\mathbf{F}(\mathbf{x}) =
\begin{pmatrix}
  \mathbf{f}_{1}(\mathbf{x})\\
\vdots \\
\mathbf{f}_{m}(\mathbf{x})
\end{pmatrix}
=
\begin{pmatrix}
  f_{11}(\mathbf{x}) & \ldots & f_{1d}(\mathbf{x}) \\
  \vdots & \ddots & \vdots \\
  f_{m1}(\mathbf{x}) & \ldots & f_{md}(\mathbf{x}) \\
\end{pmatrix}
\end{eqnarray} 
  is defined componentwise; 
$(\partial_e \mathbf{F}(x))_{1\le i \le m, 1 \le j \le n} =
(\partial_e {f}_{ij}(x))_{1\le i \le m, 1 \le j \le n}$ for all
$e\in S^{d-1}$ and all $x \in \R^d$.
      The gradient of a $m\times r$ matrix valued function $\mathbf{F}$ of the form \eqref{fform} 
     {is defined} as the $d \times r \times m$ tensor with entry
      $(\nabla F)_{i, j, k} = \partial_i f_{jk}$,
      $1 \le i \le d, 1 \le j \le m, 1 \le k \le r$.
}
The divergence of
{$\mathbf{F}$}
   as in \eqref{fform} is defined as the $m \times 1$ column vector with
components $\mathrm{div}(\mathbf{f}_j), j= 1, \ldots, m$; so that
\begin{align*}
  \mathrm{div}(\mathbf{F})
 & = 
  \begin{pmatrix}
 \mathrm{div} (\mathbf{f}_{1}) \\
\vdots  \\
    \mathrm{div}(\mathbf{f}_{m})
  \end{pmatrix}
 =
\begin{pmatrix}
  \sum\limits_{i=1}^d \frac{\partial f_{1i}}{\partial x_i} \\
 \vdots \\
\sum\limits_{i=1}^d \frac{\partial f_{mi}}{\partial x_i}
\end{pmatrix}.
\end{align*}
Similarly, the divergence of
$\mathbf{F}= \left( \mathbf{f}_1, \ldots, \mathbf{f}_m \right)$ with
values in $\R^d \times \R^m$ is
$\mathrm{div}(\mathbf{F}) = (\mathrm{div}(\mathbf{F}^T))^T$.
\subsubsection*{The product rules}

Given $v$ and $w$ two sufficiently smooth functions from
$\R^d \to \R$ and $e \in S^{d-1}$,  the directional derivative satisfies the \emph{product
  rule}
\begin{equation}
  \label{eq:46}
  \partial_e \left( v(x) w(x) \right) = (\partial_e v(x)) w(x) + v(x) (\partial_ew(x),  \quad
 \mbox{ or short: } \partial_e (v w) = (\partial_e v) w + v \partial_e w, 
\end{equation}
for all $x \in \R^d$ at which all derivatives are defined.  Gradients and divergences therefore also
satisfy the corresponding product rules. {We shall mainly consider three instances.}
\begin{enumerate} 
\item If
$\mathbf{v}$ is a $m$-dimensional vector field and
$\phi : \R^d \to \R$ then
\begin{equation}
  \label{eq:5}
  \nabla (\phi \mathbf{v}) = \phi (\nabla \mathbf{v})  +  (\nabla \phi) \mathbf{v}
\end{equation}
(a $d \times m$ matrix if $\mathbf{v}$ is a line, a $m\times d$ matrix
if $\mathbf{v}$ is a column). In particular if $m=d$ and $\mathbf{v}$ is a $d \times 1$ column vector, then
\begin{equation}\label{productrule}
 \mathrm{div}(\phi \mathbf{v})=     \phi \, \mathrm{div}
 (\mathbf{v}) + \left\langle \nabla \phi,  \mathbf{v} \right\rangle
\end{equation}   
(a scalar). 
When $\mathbf{v} = \nabla \psi$ is a gradient of a
sufficiently smooth function $\psi$, then
\begin{equation}
  \label{eq:productlapla}
   \mathrm{div}(\phi \nabla \psi )=     \phi \, \Delta \psi
 + \left\langle \nabla \phi,  \nabla \psi \right\rangle. 
\end{equation}
\item
For $\mathbf{v}\in \R^m$ and
$\mathbf{F}
$ a $m\times d$ square matrix
we have
\begin{equation}\label{eq:39}
  \mathrm{div}(\mathbf{v}^T \mathbf{F}) =   \left\langle
  \mathbf{v},  \mathrm{div}\mathbf{F}\right\rangle + \left\langle  \nabla
\mathbf{v},  \mathbf{F} \right\rangle_{\mathrm{HS}}
\end{equation}
(a scalar).  
\item 
If
$\mathbf{v}, \mathbf{w} \in \R^m$ and $A$ is a $m\times m$ positive
definite symmetric matrix then (with a slight abuse of notation)
\begin{equation}
  \label{eq:gradientproductrule}
 \nabla \left(\left\langle \mathbf{v}, \mathbf{w} \right\rangle_A \right)     =
 \left\langle \nabla 
  \mathbf{v}, \mathbf{w} \right\rangle_A + 
  \left\langle \mathbf{v},  \nabla
      \mathbf{w}  \right\rangle_A.
\end{equation} 
\end{enumerate} 


\section{Stein  operators and their standardizations}
\label{sec:first-defin-outl}

In this paper, all random vectors are assumed to admit a {probability density function (pdf)} 
$p$ with respect to the Lebesgue measure. {The support of $p$, denoted by $K_p$ , is} 
 the complement of the largest open set
$\mathcal O$ such that $p \equiv 0$ on $\mathcal O$. We further make
the following assumption:

\

\noindent \textbf{Assumption A}: There exists an open
set $\Omega_p$ such that the Lebesgue measure of
$K_p \backslash \Omega_p$ is zero, $p>0$ on $\Omega_p$ and $p$ is
$\mathcal C^1$ on $\Omega_p$.

\

If $p$ satisfies this assumption, then $p$ is $\mathcal C^1$
on $\Omega_p \cup K_p^C$,
which by Assumption A is a set of full Lebesgue measure.  For ease of
notation the Lebesgue measure is left out in integrals, so that
$\int_{\R^d} f(x) dx $ is written as $\int_{\R^d} f$. Similarly we
abbreviate $\E_p f = \int_{K_p} f p = \int_{K_p} f(x) p(x) \, dx$.

\subsection{Definition and some general comments}
\label{sec:general-theory}

\begin{definition}[Multivariate Stein class and operator]\label{def:steclass}
  Let $X$ be a $d$-dimensional random vector with pdf $p : \R^d \to \R$ satisfying Assumption A.
  \begin{enumerate}
  \item The \emph{canonical directional Stein class for
        $p$} {in direction $e$} is the vector field
      $\mathcal{F}_{1,e}(p)$ of all functions $\phi : \R^d \to \R$
      such that the restriction of $\phi$ to $\Omega_p$ is
      $\mathcal C^1$, $p\, \phi$ has integrable gradient and satisfies
      $ \int_{\R^d} \partial_e (p\, {\phi}) = {0}$.
  \item The \emph{canonical Stein derivative for $p$ in the direction $e$}
  is the operator $ \mathcal{T}_{e, p}$ on $\mathcal{F}_{1,e}(p)$
  defined by $\mathcal T_{e,p}\phi = \frac{\partial_e( p \, 
   \phi)}{p}
$. 
\end{enumerate} We set $\mathcal{F}_1(p)$ the collection of all
    scalar valued functions $f$ which belong to
    $\mathcal{F}_{1, e_i}(p)$ for all unit vectors
    $e_i, i=1, \ldots, d$, and define the Stein class of $p$, denoted
    $\mathcal{F}(p)$, as the collection of all scalar, vector and
    matrix valued functions whose components belong to
    $\mathcal{F}_1(p)$.
   \end{definition}
  Note in particular that if $\phi \in \mathcal F_1(p)$, then $\int _{\R^d} \partial_e(p \, \phi) = 0$ for all unit vectors $e$.
  {Furthermore,} any  function $\phi\in\mathcal{C}^1(\R^d)$ with compact support $K \subset \Omega_p \cup K_p^C$ lies in $\mathcal F_{1,e}(p)$. Indeed, if $e = (a_1,\ldots,a_d)$, since $p\phi \in \mathcal C^{1}(\R^d)$, we have $\partial_e (p\phi) = \sum_{i=1}^d a_i \partial_i(p\phi)$, then
  $$\int_{\R^d} \partial_i( p \phi) = \int_{\R^{d-1}} \left( \int_\R \partial_i(p \phi) dx_i\right)\prod_{j\neq i} dx_j,$$
  and the inside integral is zero since $p(x)\phi(x)$ cancels out for $x$ away enough from $0$.

  {On $\mathcal{F}(p)$ we define the \emph{Stein gradient operator}
    acting on real valued functions $f$, vector valued functions
    $\mathbf{f}$ or matrix valued functions $\mathbf{F}$, as
   $$\bullet \mapsto \mathcal{T}_{\nabla, p}  \bullet = \frac{\nabla \left( p \bullet
     \right)}{p}$$ (expressed in terms of $\mathcal{T}_{e, p}$'s as
   $ \mathcal{T}_{\nabla, p} = \sum_{i=1}^d e_i\mathcal{T}_{e_i, p}$);
   we also define the \emph{Stein divergence operator} acting on
   vector or matrix valued functions with compatible dimensions as  
  $$\bullet \mapsto \mathcal{T}_{\mathrm{div}, p}  \bullet =
  \frac{\mathrm{div}\left( p \bullet 
    \right)}{p}$$ (this operator also can be expressed in terms of
  $\mathcal{T}_{e, p}$'s). 
  Although {in the literature, it is} not $\mathcal{T}_{e, p}$ but 
    $\mathcal{T}_{\nabla, p}$ and $ \mathcal{T}_{\mathrm{div}, p}$
    that are the operators used,
  their properties follow 
  from those of {$(\mathcal{T}_{e, p}), e \in \mathcal{B}$ with $\mathcal{B}$ a  unit basis of $\R^d$}.
\begin{rem}\label{rem:stress0}
  Let $\mathcal{F}^{(0)}(p)$ denote the
  collection of functions (in any dimension) with mean 0 under $p$.
  The definition of $\mathcal{F}(p)$ is tailored to ensure
  $\mathcal{T}_{\mathcal{D}, p} (\mathbf{F}) \in \mathcal{F}^{(0)}(p)$
  for all $\mathbf{F} \in \mathcal{F}(p)$ (here and below
  $\mathcal{D}$ denotes $\partial_e, \nabla$ or $\mathrm{div}$). Hence, the Stein
  operators are also a machinery for producing mean 0 functions under
  $p$. This can be useful in applications, particularly when $p$ is
  intractable, since by construction the operators are invariant under
  scalar multiplication, and therefore do not depend on normalizing
  constants. 
\end{rem}

\subsubsection*{Stein identities}

Product rule \eqref{eq:46} directly leads to Stein-type product rules {of the type} 
\begin{equation}
  \label{eq:24}
  \mathcal{T}_{e, p}(f\, g) =   (\mathcal{T}_{e, p}f) g +
  f (\partial_e g ) 
\end{equation}
 {on a suitable set of functions.}
The next
    definition introduces such sets.

\begin{definition}[Stein adjoint class]
  To every (scalar/vector/matrix valued function)
    $\mathbf{F}\in \mathcal{F}(p)$ we associate
    $ \mathrm{dom}(p, \mathbf{F})$ the collection of all matrix-valued
    functions $\mathbf{G}$ with compatible dimensions such that $\mathbf{G}$ is $\mathcal C^1$ on $\Omega_p$, 
    $\mathbf{F} \mathbf{G} \in \mathcal{F}(p)\mbox{ and } \mathbf{F}
    (\partial_e \mathbf{G}) \in L^1(p)$ for all $e\in S^{d-1}$. 
\end{definition}

Using the
  product rules 
  from Section \ref{sec:notations-1}, we deduce similar identities for
  Stein gradients and divergences.  {Here we give two instances.} 
  \begin{enumerate}
  \item For all scalar functions  $f \in {\mathcal{F}}_1(p) $
      and $g \in \mathrm{dom}(p, f)$,  from \eqref{eq:5} 
\begin{align}
& \mathcal{T}_{\nabla, p} \left( f g\right) =
 ( \mathcal{T}_{\nabla, p} f) g + f  \nabla g. 
  \label{eq:28}
\end{align}
\item  For all  $\mathbf{F}:\R^d\to\R^m\times \R^d \in {\mathcal{F}}_1(p) $
    and
    $\mathbf{g}:\R^d\to\R\times \R^m\in \mathrm{dom}(p,
    \mathbf{F})$,  using \eqref{eq:39},
\begin{align}
  \label{eq:9}
  &  \mathcal{T}_{\mathrm{div}, p}(\mathbf{g}\mathbf{F}) =
   \left\langle
 \mathcal{T}_{\mathrm{div}, p} \mathbf{F}, \mathbf{g} \right\rangle   + 
\left\langle \mathbf{F},\nabla \mathbf{g} \right\rangle_{\mathrm{HS}}. 
\end{align}
 We deduce a family of Stein operators for $p$ by
fixing
$\mathbf{F} = \left( F_{ij} \right)_{1\le i \le m, 1 \le j \le d}$ and
considering $\mathcal{A}_p := \mathcal{A}_{\mathbf{F}, p}$ defined by
\begin{align}\label{eq:74}
\mathbf{g} \mapsto   \mathcal{A}_p  \mathbf{g} =  \left\langle
 \mathcal{T}_{\mathrm{div}, p} \mathbf{F}, \mathbf{g} \right\rangle   + 
\left\langle \mathbf{F},\nabla \mathbf{g} \right\rangle_{\mathrm{HS}}
\end{align}
with domain
$\mathcal{F}( \mathcal{A}_p) := \mathrm{dom}(p, \mathbf{F}) $. Note
that $\mathcal{A}_p(\mathbf{g}) $ is a scalar function.  We also
deduce the Stein identities (one for each $\mathbf{F}$)
\begin{equation}
  \label{eq:25}\E_p \left[  \left\langle
 \mathcal{T}_{\mathrm{div}, p} \mathbf{F}, \mathbf{g} \right\rangle \right]   = -\E_p \left[ 
   \left\langle \mathbf{F},\nabla \mathbf{g}
   \right\rangle_{\mathrm{HS}}  \right] \mbox{ for all } \mathbf{g}
 \in \mathcal{F}(\mathcal{A}_p). 
\end{equation} 
More 
product rules are provided
later, see e.g.\  \eqref{eq:31},
\eqref{eq:prodrulenabla}.
\end{enumerate}


It follows from the
    definition of {the  classes of functions $ \mathrm{dom}(p, \mathbf{F})$ and} the Stein operators and the Stein class that the
    left hand side in each of \eqref{eq:24}, \eqref{eq:28}, and their variations such as
\eqref{eq:9}, \eqref{eq:74}, \eqref{eq:23}, \eqref{eq:31}, integrates to 0. These
probabilistic integration by parts formulas lead to identities known
as \emph{Stein (covariance) identities}; they are  inspired by
Stein's original Gaussian identity \eqref{eq:stlemma}.

\begin{prop}[Stein identities]\label{prop:intbp}
  For all $f {: \R^d \to \R} \in \mathcal{F}_1(p)$ 
  and all $e\in S^{d-1}$ we have
\begin{equation}
  \label{eq:10}
 \E_p \left[ f (\mathcal{T}_{e, p}g)   \right] = -
 \E_p \left[g (\partial_e f)  \right]
\end{equation}
for all $g {: \R^d \to \R} \in \mathrm{dom}(p, f)$
\end{prop}
\begin{proof}
{Let $f \in \mathcal{F}_1(p)$ and $g \in \mathrm{dom}(p, f)$ and
  $e\in S^{d-1}$. Then, from \eqref{eq:24},
  $ f (\mathcal{T}_{e, p}g) = \mathcal{T}_{e, p}(f\, g) - g \partial_e
  f $. Since  $\E_p \mathcal{T}_{e, p}(f\, g) =0$ for all
  $f \in \mathcal{F}_1(p)$ and $g \in \mathrm{dom}(p,
  f)$, identity \eqref{eq:10} ensues.}
\end{proof} 

Identity \eqref{eq:10} {shows that} 
the
Stein operators  are \emph{the} skew adjoint operators to the classical directional
derivatives, with
respect to integration in $p$. In this sense
the operator from Definition~\ref{def:steclass} is ``canonical''.

  The Stein identities \eqref{eq:10} for $(e_1, \ldots, e_d)$ the
  standard unit basis in $\R^d$  yield corresponding Stein
  identities for the gradient Stein operator $\mathcal{T}_{\nabla, p}$
  and for the divergence Stein operator
  $ \mathcal{T}_{\mathrm{div}, p}$. {Here we give three instances.}
  \begin{enumerate}
  \item  For all 
  $f \in \mathcal{F}_1(p)$ and $ g \in \mathrm{dom}(p, f)$, 
  \begin{equation}
    \label{eq:11}
   \E_p \left[ (\mathcal{T}_{\nabla, p} f) g
  \right] = - 
    \E_p \left[ f (\nabla g) \right] `
  \end{equation}
  \item  For all $f \in \mathcal{F}(p)$ and $g$ such that $\nabla g\in \mathrm{dom}(p, f)$,
   from \eqref{eq:productlapla} we
  obtain 
  \begin{equation}
    \label{eq:23}
    \mathcal{T}_{\mathrm{div}, p}(f \nabla g) = f \Delta  g
    + \left\langle \mathcal{T}_{\nabla, p}f, \nabla  g \right\rangle. 
  \end{equation}
Consequently, 
\begin{equation}
  \label{eq:35}
\E_p \left[  f \Delta  g \right] = -\E_p \left[ 
    \left\langle \mathcal{T}_{\nabla, p}f, \nabla  g \right\rangle \right]
\end{equation}
for all such $g$.  
\item 
For all $m\times d$ square matrices $\mathbf{F}\in \mathcal{F}(p)$ and  $\mathbf{g}:\R^d\to\R^m\in \mathrm{dom}(p, \mathbf{F})$,   using \eqref{eq:39} gives 
  \begin{equation}
    \label{eq:31}
    \mathcal{T}_{\mathrm{div}, p}(\mathbf{g}^T \mathbf{F})  =
   \left\langle \mathbf{g}, \mathcal{T}_{\mathrm{div}, p}(\mathbf{F})  \right\rangle+
   \left\langle \nabla \mathbf{g}, \mathbf{F} \right\rangle_{\mathrm{HS}}. 
  \end{equation}
 Consequently,
\begin{equation}
  \label{eq:41}
 \E_p \left[ \left\langle \mathbf{g},  \mathcal{T}_{\mathrm{div},
        p}(\mathbf{F}) \right\rangle \right] = - 
 \E_p \left[ \left\langle \nabla \mathbf{g},  \mathbf{F}  \right\rangle_{\mathrm{HS}} \right] .
\end{equation}
\end{enumerate} 

\subsubsection*{Stein characterizations}

Recall that $\E_p \left[ \mathcal{T}_{e, p} \mathbf{F} \right] =0$ for
every $ \mathbf{F}\in \mathcal{F}(p).$ Under well chosen assumptions, {the collection of operators 
$\mathcal{T}_{e_i, p}$ for  $(e_1, \ldots, e_d)$ the standard unit basis of $\R^d$ characterises $p$ in the sense that} 
if
$\E_q \left[ \mathcal{T}_{e_i, p} \mathbf{F} \right] = 0, i=1,
    \ldots, d$ for a sufficiently wide class of $\mathbf{F}$,
then
necessarily $q = p$.

\begin{theorem}[Stein characterizations]\label{sec:diff-stein-oper-directional}
  Let $X \sim p$ and $Y\sim q$; assume that $p$ and $q$ both satisfy
  Assumption A. Assume moreover that $\Omega_p = \Omega_q$, and that
  this set is connected. Let $f: \R^d \rightarrow \R \in
      \mathcal{F}_1(p)$, and assume ${f}>0$ on $\Omega_p $. 
      \begin{enumerate}
      \item It holds that $Y
  \stackrel{\mathcal{L}}{=} X$ if and only if for all{ {
      $\mathbf g : \R^d \rightarrow \R^d \in \mathrm{dom}(p, {f})$}}  \begin{equation}
  \label{eq:steskewidendirectional}
   \E_q \left[ (\mathcal{T}_{\dv, p} \mathbf g) f
  \right] = - 
    \E_q \left[ \langle  \mathbf g , \nabla f \rangle\right] .
  \end{equation}
      \item \label{sec:diff-stein-oper}
It holds that $Y \stackrel{\mathcal{L}}{=} X$
  if and only if for all
  $g: \R^d \rightarrow \R \in \mathrm{dom}(p,f)$
  \begin{equation}
  \label{eq:steskewiden}
   \E_q \left[ (\mathcal{T}_{\nabla, p} g) f
  \right] = - 
    \E_q \left[ g (\nabla f) \right] .
  \end{equation}
      \end{enumerate}
\end{theorem}
  
\begin{proof}
First let $f: \R^d \rightarrow \R \in \mathcal{F}_1(p)$ and $\mathbf g : \R^d \rightarrow \R^d \in \mathrm{dom}(p, {f})$. 
By the product rule \eqref{productrule}, 
$ f \dv (p \mathbf g)  = \dv(f p \mathbf g )  -  p \langle \nabla f, \mathbf g\rangle$
and hence 
$ f \mathcal{T}_{\dv, p} \mathbf g = \mathcal{T}_{\dv, p} (f  \mathbf g )  -  \langle \nabla f, \mathbf g\rangle.$
Taking expectations gives 
$$ \E_p [ f \mathcal{T}_{\dv, p} \mathbf g ]  =  \E_p [
\mathcal{T}_{\dv, p} (f \mathbf g ) ] - \E_p [\langle \nabla f,
\mathbf g\rangle].$$ By construction of $\mathrm{dom}(p, f)$,
$\E_p [\mathcal{T}_{\dv, p} (f \mathbf g )] =0$ and the only-if
direction follows.  

For the converse direction, let
$\mathbf g$ be an infinitely differentiable vector field with
compact support $K \subset\Omega_p$; recall that in this case $g \in\mathrm{dom}(p, {f})$.
Suppose that {\eqref{eq:steskewidendirectional} holds.} 
Re-writing the expectations, 
$$ \int_K \dv (p \mathbf g) f \frac{q}{p} = 
- \int_K q \langle \nabla f, \mathbf g\rangle
.
$$
Now, by the divergence (or Ostrogradski) theorem, and since $\mathbf g \equiv 0$ on $\partial K$,
$$ \int_K \dv (p \mathbf g) f \frac{q}{p} =- \int_K p \,\bigg\langle
\mathbf g, \nabla \left( f\frac{q}{p} \right)\bigg\rangle.$$ Thus
$\int_K p \,\bigg\langle \mathbf g, \nabla \left( f\frac{q}{p}
\right)\bigg\rangle = \int_K \langle \mathbf g, \nabla g\rangle
q$. Hence $p \nabla \left( f\frac{q}{p} \right) = q\nabla g$ in the
weak sense on $\Omega_p$; but it is also true in the strong sense,
because all functions are continuous (recall that
$\Omega_p=\Omega_q$). We deduce that
$f \nabla(q/p) = 0, \quad \text{on } \Omega_p,$
but since $f >0$ on this connected domain, $q/p$ is constant on
$\Omega_p$. Finally, using $\int_{\Omega_p} p = \int_{\Omega_p} q$ gives  the first result.

For Item \ref{sec:diff-stein-oper},  proceeding as for the first part, we arrive at 
 the differential equation  
$  p \nabla \left( f \frac{q}{p} \right) = q\nabla f$
 from which 
 we deduce that $p/q = 1$. 
\end{proof} 
  

The proof of the claims in Theorem \ref{sec:diff-stein-oper-directional} are
greatly facilitated by the tailored assumptions which hide some
difficulties. It still remains to write such characterizations out in
detail for specific targets; this can turn out to be quite complicated
and will not be the focus of the paper. Also, 
there are many
more general ways to formulate similar characterizations as in Theorem
\ref{sec:diff-stein-oper-directional}
 be it by changing the starting operator, by
relaxing the assumptions on the test functions or indeed by relaxing
the assumptions on $Y$.

\subsection{Standardizations}
\label{sec:standardizations}

{Similarly as in the univariate case (see \cite[Section
  4]{LRS16}),} we introduce a broad family of Stein operators
$\mathcal{A}_p : \mathcal{F}(\mathcal{A}_p) \to \mathcal{F}^{(0)}(p)$
which we call \emph{standardizations} of the canonical operator
$\mathcal{T}_{\mathcal{D}, p}$, where $\mathcal{D} = \nabla$ or
${\rm{div}}$. These are operators such that there exists a
transformation
$\mathbf{T} : \mathcal{F}(\mathcal{A}_p) \to \mathcal{F}(p) :
\mathbf{u} \mapsto \mathbf{T}(\mathbf{u})$ such that
$\mathcal{A}_p(\mathbf{u}) = \mathcal{T}_p(\mathbf{T}(\mathbf{u}))$
for all $\mathbf{u} \in \mathcal{F}(\mathcal{A}_p)$.  Such
standardizations were studied in \cite{LRS16}; \cite[Appendix
A.2]{oates2017posterior} suggests some standardizations for densities
on a manifold. Here we concentrate on standardizations which are
connected to score functions and Stein kernels.

\subsubsection{Gradient based operators and the score function}
 \label{sec:grad-based-oper}
 
 In this section we focus on the Stein identity \eqref{eq:11}, 
which
 arises from the product rule \eqref{eq:28} extended as
\begin{equation}
  \label{eq:prodrulenabla}
  \mathcal{T}_{\nabla, p}(f \mathbf{g}) =
(\mathcal{T}_{\nabla, p}  f) \mathbf{g}  + f \nabla \mathbf{g}
\end{equation}
for $f\in \mathcal{F}_1(p)$ and
$\mathbf{g}: \R^d \rightarrow \R^m \in \mathrm{dom}(p,f)$.
From this, we deduce a family of Stein operators for $p$ obtained by
fixing some $f\in \mathcal{F}(p)$ and considering the operator
$ \mathcal{A}_p := \mathcal{A}_{f, p}$ given in
\eqref{eq:prodrulenabla} acting through
$\mathbf{g} \mapsto \mathcal{A}_p \mathbf{g}:=\mathcal{T}_{\nabla,
  p}(f \mathbf{g})$ on
$\mathbf{g}:\R^d\to \times \R^{m} \in \mathcal{F}(\mathcal{A}_p) =
\mathrm{dom}(p, f)$.  Note that $\mathcal{A}_p(\mathbf{g})$ is a
$d \times m$ matrix, and
$\E_p [\mathcal{A}_p \mathbf{g}] = \mathbf{0}$ for all
$\mathbf{g} \in \mathrm{dom}(p,
f)$.  
Each particular $f \in \mathcal{F}_1(p)$ thus gives rise to a Stein
identity
\begin{equation}
  \label{eq:8}\E_p \left[  (\mathcal{T}_{\nabla, p}  f)\mathbf{g}  \right] = -\E_p
\left[  f \nabla \mathbf{g}\right] \mbox{ for all } \mathbf{g} \in
\mathcal{F}(\mathcal{A}_p). 
\end{equation}
%
 One particular {choice for $f$}  stands out:
$f = 1$. The following definition is classical.
\begin{definition}[Score function] \label{def:multiscore}
  Let $p$ be differentiable. The \emph{score function} of $p$ is the function defined on $\Omega_p$ by
\begin{equation}
  \label{eq:defscoregen}
  \rho_p(x) = \nabla \log p(x) = \frac{1}{p(x)} \nabla p(x).
\end{equation}
The \emph{score-Stein operator} is the vector valued operator
\begin{equation}
\label{eq:1}
      \mathcal{A}_pg = \nabla g  +  \rho_p \, g
    \end{equation}
    acting on differentiable functions $g : \R^d \to \R$.
\end{definition}
  \begin{rem}
    One can easily extend Definition \ref{def:multiscore} to
    introduce  a score (and score-Stein operator) in any direction
    $e\in S^{d-1}$, by considering the gradient $\nabla = (\partial_e,
    \partial_{e^{\perp}})$ along $e$.  
\end{rem}
Equation \eqref{eq:1} is often called \emph{the} Stein operator for
multivariate $p$; see e.g.\ \cite[definition 2.1]{liu2016kernelized},
\cite[Section 2.1]{chwialkowski2016kernel}, \cite[Section
2.2.2]{oates2017control}.  This operator is particularly interesting
when constant functions $1 \in \mathcal{F}_{{1}}(p)$, an assumption
which holds if $p$ is a differentiable density such that
$\partial_i p$ is integrable for all $i = 1, \ldots, d$ and
$\int_{\R^d} \partial_ip = 0$.
It is easy to see that
$\mathcal{C}_0^{\infty}(\R^d) \subset \mathrm{dom}(p, 1)$. The
resulting (characterizing) Stein identity is
 \begin{equation}
   \label{eq:3}
\E_p \left[ \rho_p \, g
   \right] = -  \E_p \left[ \nabla g \right] \mbox{ for all }g \in \mathcal{F}(\mathcal{A}_p).
 \end{equation}
 Equation \eqref{eq:3} is reminiscent of the classical Stein identity
 \eqref{eq:stlemma}; one can easily check that
 $\rho_{\phi}(x) = \nabla \log \phi(x) = -\Sigma^{-1} (x-\nu)$ when
 $X\sim \phi$ the multivariate normal density with mean $\nu$ and
 covariance $\Sigma$. 
{We shall see that}  for the Student-t distribution  the choice of a non-constant $f$ in \eqref{eq:8} leads to
 more tractable operators and identities.  
 We will discuss operator
 \eqref{eq:1} and Stein identity \eqref{eq:3} -- and variations
 thereon -- in the more general context of score functions and Fisher
 information in Section \ref{sec:kern-stein-discr}.

\subsubsection{Divergence based first order operators and Stein kernels}
\label{sec:diverg-based-first}

In this section, we focus on the product rule \eqref{productrule}
rewritten as
\begin{align}
\label{eq:42}
  \mathcal{T}_{\mathrm{div}, p}(\mathbf{F}g) =
  (\mathcal{T}_{\mathrm{div}, p} \mathbf{F}) g + \mathbf{F} \nabla g
\end{align}
for all properly chosen $\mathbf{F}:\R^d\to\R^m\times \R^d$ and all
$g : \R^d\to \R$. By construction,
$\mathcal{T}_{\mathrm{div}, p}(\mathbf{F}g)\in \R^m$ for all
$\mathbf{F}, g$.  From this we deduce a family of Stein operators for
$p$ obtained by fixing $\mathbf{F}$ in $\mathcal{F}(p)$ and
considering $\mathcal{A}_p := \mathcal{A}_{\mathbf{F}, p}$ defined by
${g} \mapsto \mathcal{A}_p {g} = (\mathcal{T}_{\mathrm{div}, p}
\mathbf{F}) g + \mathbf{F} \nabla g $ with domain
$\mathcal{F}( \mathcal{A}_p) := \mathrm{dom}(p, \mathbf{F}) $.  We
also deduce the Stein identities (one for each $\mathbf{F}$)
\begin{equation}
\label{eq:37}\E_p \left[   
(\mathcal{T}_{\mathrm{div}, p} \mathbf{F}) {g}  \right]   = -\E_p \left[ 
   \mathbf{F}\nabla {g} \right] \mbox{ for all } {g} \in \mathcal{F}(\mathcal{A}_p).
\end{equation}

One particularly important choice of $\mathbf{F}$ {in \eqref{eq:37}} 
is the \emph{Stein kernel} first
defined in dimension $d=1$ in \cite{stein1986}
. Here we consider general dimension $d \ge 1$ and
propose the following definition.

\begin{definition}[Stein kernel]\label{def:multiteikdef}
  Consider a density $p : \R^d \to \R^+$ satisfying Assumption A, and
  suppose $p$ has finite mean $\nu \in \R^d$. For each canonical
  direction $e_i \in \R^d$, $i=1, \ldots, d$, \emph{a Stein kernel}
  for $p$ \emph{in the direction $e_i$} is any vector field
  $x \mapsto \tau_{p,i}(x) \in \R^d$ belonging to $\mathcal{F}(p)$
  such that
  \begin{equation}
\label{eq:65}
    \mathcal{T}_{\mathrm{div}, p}\big(\tau_{p, i}(x)\big) = \nu_i- x_i
  \end{equation}
  for all $x \in \Omega_p$. A \emph{Stein kernel} is any square matrix
  $ \pmb\tau_p =(\tau_{i,j})_{1 \le i,j \le d}$ such that each line
  $ \pmb \tau_i = (\tau_{i1}, \ldots, \tau_{id})$ is a kernel in the
  direction $e_i$.  The $\pmb\tau_p$-\emph{kernel-Stein operator} is
  the $\R^d$-valued operator
\begin{equation}
  \label{eq:29}
  \mathcal{A}_pg(x) =  \pmb\tau_p(x) \nabla g(x)- (x-\nu)g(x)
\end{equation}
with domain $\mathcal{F}(\mathcal{A}_p) = \mathrm{dom}(p, \pmb\tau_p)$.  
  \end{definition} 
  
  When the kernel is clear, we leave out the $\pmb\tau_p$-prefix in
  the kernel-Stein operator.  
  {To illustrate the concept we give two examples; Section \ref{sec:stein-kernel} contains more 
constructions for some cases.} 
  
  \begin{example}  \label{sec:diverg-based-first-1}
A Stein kernel in the sense
of Definition \ref{def:multiteikdef} is immediately obtained in the following way. For each $i=1, \ldots, d$ let 
$\tau_{p, i} = (0, \ldots, \tau_i, \ldots, 0)$ with 
\begin{equation*}
  \tau_i(x) = \frac{1}{p(x)} \int_{x_i}^{\infty} (\nu_i - u_i)
  p(x_{\setminus i}, u_i) du_i
\end{equation*}
where
$(x_{\setminus i}, u_i) = (x_1, \ldots, x_{i-1}, u_i, x_{i+1}, \ldots,
x_d)$. The diagonal matrix $\pmb \tau$ obtained by stacking up the
above vectors is a Stein kernel. If $X$ has finite first moment, this
Stein kernel is defined for almost every $x$ by Fubini's theorem. Such
diagonal kernels are often not interesting (mainly because the Stein
identity \eqref{eq:defstek2} {below} will not hold for a large
enough class of functions).
  \end{example}
  
\begin{example}\label{contstruction}
  The score function and the divergence operator are linked through 
  $$\mathcal{T}_{\mathrm{div, p}} \mathbf{F} = \mathbf{F} \rho_{p} +
  \mathrm{div} (\mathbf{F}) ,$$ where
  $\mathbf{F} \in {\mathcal{F}}(p)$.  Hence for a density $p$ with
  mean $\nu$, a Stein kernel is any matrix $\pmb \tau$ satisfying
  $$ {\pmb \tau}(x) \rho_p (x) + \mathrm{div}({\pmb \tau}(x)) =  -
  x$$ for all $x\in\Omega_p$.  A general construction {of} Stein
  kernels is thus given by finding an
  $\mathbf{F} \in {\mathcal{F}}(p)$ such that for some constants
  $\alpha, \beta \in \R$, $\alpha + \beta \ne 0$,
  \begin{eqnarray}
  \mathbf{F}(x)  \rho_{p} (x) &= & \alpha (x - \nu) + r(x) \label{constr1} \\
  \mathrm{div}(  \mathbf{F}(x) ) &=& \beta ( x - \nu) - r(x) \label{constr2} 
  \end{eqnarray}
($r(x)$ is \emph{any} function) and setting 
  $$ {\pmb \tau} = \frac{1}{\alpha + \beta}   \mathbf{F}.$$
  \end{example}
 
 \begin{rem}
 \begin{enumerate}
 \item 
  With Definition \ref{def:multiteikdef}, existence of the Stein
  kernel is readily checked, and the resulting Stein identity is
 \begin{equation}
  \label{eq:defstek2}
  \E \left[ \pmb{\tau}(X)\nabla g(X)
  \right] = \E \left[    (X-\nu) g(X)\right] \mbox{ for all } g \in
  \mathcal{F}(\mathcal{A}_p) = \mathrm{dom}(p, \pmb \tau). 
\end{equation}
  Equation \eqref{eq:defstek2} is
reminiscent of the classical Stein identity \eqref{eq:stlemma}; one
can easily check that $\pmb \tau_p(x) = \Sigma$ is a Stein kernel for the
multivariate standard Gaussian distribution.  
\item 
The question of whether a density $p$ admits a Stein kernel is non
trivial, if one wants the Stein identity to hold for a larger class of
functions than the regular ones with compact support in $\Omega_p$ (see
Section \ref{sec:stein-kernel}).
Definition \eqref{def:multiteikdef} {agrees with} the one of
\cite{courtade2017existence}, except for the class of test functions
we impose.
 %
\item 
The Stein kernel $\tau_{p,i} \in \mathcal F(p)$ in the
direction $e_i$ in Definition \eqref{def:multiteikdef} could be
defined equivalently by requiring
\begin{equation}
  \label{eq:steinkernelidentity} \E \left[ \langle \tau_{p,i}(X),
  \nabla g(X) \rangle \right] = \E \left[ (X_i-\nu_i) g(X)\right],
\end{equation}
for all $g \in \mathcal{C}^\infty_c(\Omega_p) $, the set of functions infinitely differentiable with compact support
in $\Omega_p$. Thus, showing some matrix is a Stein kernel can be done
equivalently by checking pointwisely \eqref{eq:65}, or by checking the
previous identity holds at least for infinitely differentiable functions with compact
support included in $\Omega_p$ (and that it is in $\mathcal F(p)$). However this class of functions is narrow, and
in most applications, one would want \eqref{eq:steinkernelidentity} to
hold for more test functions $g$ -- e.g. for bounded $g$. In this
work, we use the above definition to find Stein kernels, and we extend
the class of functions for which the identity holds on a case-by-case
basis. This will be done, for instance, for the Stein kernels we build
for elliptical distributions. From the divergence
theorem, when the boundary of $\Omega_p$ is smooth enough, any
$g\in \mathcal C^1(\Omega_p)$ such that $\tau g p \rightarrow 0$ on
the boundary of $\Omega_p$ will verify the identity
\eqref{eq:steinkernelidentity}.

    \end{enumerate}
    \end{rem}

\subsubsection{Divergence based second order operators}
  \label{sec:diverg-based-second}

  The starting point here is the divergence product rule \eqref{eq:39}
  and the corresponding Stein operator equation \eqref{eq:31},
  extended as follows.  Choose $\mathbf{A}, \mathbf{B}$ two
  $d \times d$ matrix valued functions in $\mathcal{F}(p)$. Then
  \begin{align}\label{eq:17}
\mathcal{T}_{\mathrm{div},p}\big((\nabla g)^T\mathbf{A}
    \mathbf{B}^T\big)   = 
\left\langle \nabla g, \mathcal{T}_{\mathrm{div, p}}(\mathbf{B}) \right\rangle_{\mathbf{A}}  +
  \left\langle  \nabla\left( (\nabla g)^T \mathbf{A}
    \right),   \mathbf{B} \right\rangle_{\mathrm{HS}}
  \end{align}
  for all $g \in\mathcal C^1(\Omega_p)$.  This leads to a family of second order
  scalar valued operators obtained by fixing $\mathbf{A}$ and/or $\mathbf{B}$
  and considering
  $ \mathcal{A}_p g := \mathcal{A}_{\mathbf{A}, \mathbf{B}, p}g$
  defined in \eqref{eq:17} with domain the collection of $g$ such that
  $ (\nabla g^T \mathbf{A} \in \mathrm{dom}(p, \mathbf{B})$. {We single out three particular choices for  $\mathbf{A},$ and $\mathbf{B}$.} 
\begin{enumerate}
 \item  $\mathbf{A} = \mathbf{B} = I_d$ yields
\begin{equation}
  \label{eq:A}
  \A_{p}  g  =  \left\langle \nabla \log p , \nabla g \right\rangle   +  \Delta g, 
\end{equation}
($\Delta$ being the Laplacian on $\R^d$) with domain
$\mathrm{dom}(\mathcal{A}_p)$;
\item $\mathbf{A} = I_d$ and $\mathbf{B} = \pmb \tau_p^T$ the
  (transpose of a) Stein kernel of $X$ yields
\begin{equation}
  \label{eq:axfnuhess}
  \mathcal{B}_p g =  \left\langle \nabla g, \nu-\bullet  \right\rangle +
  \left\langle  \nabla^2g,  \pmb
    \tau_p \right\rangle_{\mathrm{HS}}
\end{equation} 
(where $\nabla^2(f) = \nabla(\nabla f^T)$ is the Hessian of
$f$ and    $\nu$ is the mean of $X$) with domain $\mathrm{dom}(\mathcal{B}_p)$. 
\item $\mathbf{A}$ symmetric definite positive and
  $\mathbf{B}$  such that $\mathcal{T}_{\mathrm{div}, p}(\mathbf{B}) =
  \mathbf{b}$ yields
\begin{equation}
  \label{eq:axfnuhess2}
  \mathcal{C}_p g =  \left\langle \nabla g, \mathbf{b}
  \right\rangle_{\mathbf{A}} +  
  \left\langle  \nabla^2_{\mathbf{A}}g,   \mathbf{B} \right\rangle_{\mathrm{HS}}
\end{equation} 
($\nabla^2_{\mathbf{A}}g = \nabla\left( (\nabla g)^T \mathbf{A}
\right)$) with domain $\mathrm{dom}(\mathcal{C}_p)$.
 \end{enumerate}
 One recognizes in operators such as \eqref{eq:axfnuhess} and
 \eqref{eq:axfnuhess2} the infinitesimal generators of multivariate
 diffusions, see e.g.\ \cite{gorham2016measuring}.

 
\subsection{Stein equations and Stein factors}
\label{sec:bound-stein-discr}

Let $X\sim p$ with Stein canonical class and operator
$(\mathcal{F}(p), \mathcal{T}_{{\mathcal{D},}}p)$, where $\mathcal{D}$
denotes $\partial_e, \nabla$, or $\dv$, and consider some
standardization of the form
\begin{equation}
  \label{eq:geneop}
  \mathcal{A}_p : \mathcal{F}(\mathcal{A}_p) \to
  \mathrm{im}(\mathcal{A}_p) : g \mapsto \mathcal{A}_pg
\end{equation}
as detailed in Subsection~\ref{sec:standardizations}.

\begin{definition}[Stein's equations and (magic)
  factors]\label{def:steeq}
  Instate all previous notations. Let $\mathcal{H} \subset L^1(p)$ be
  a family of test functions, and suppose that
  $h - \E_ph \in \mathrm{im}(\mathcal{A}_p)$.
  The $(\mathcal{A}_p-\mathcal{H})$
  \emph{Stein equation} for $X$ is the family of differential
  equations
  \begin{equation}
  \label{eq:steqgen}
  \mathcal{A}_pg = h -  \E_ph; \quad h \in \mathcal{H}.
\end{equation}
 For a given $h$, a solution to
\eqref{eq:steqgen} is an absolutely continuous function $g_h$ such
that there exists a version of the derivatives for which
\eqref{eq:steqgen} is satisfied at all $x \in \R^d$. An
$(\mathcal{A}_p-\mathcal{H})$ \emph{Stein factor} for $X$ is
\emph{any} uniform bound on some moment of (derivatives of)
$g_h = \mathcal{A}_p^{-1}(h-{\E_ph})$ over all $h \in \mathcal{H}$ of
solutions to Stein equations \eqref{eq:steqgen}.
\end{definition}

{Solving, and bounding the solution of,} the
Stein equation \eqref{eq:steqgen} 
{is well studied} in the univariate case under quite general
conditions on $p$, see for example
\cite{stein1986,ChGoSh11,NP11, Do14,dobler2015iterative} {and
\cite{ernstswan2019}}.  Matters are much more complicated in the
multivariate setting. When $p$ is a multivariate Gaussian,
then \cite{Ba90} identified a solution of the Stein equation
\eqref{mvnstein} to be given by the Mehler formula \eqref{mehler}.
Such explicit dependence of solution $f_h$ on the function $h$ permits
to study regularity properties of $f_h$ in terms of those of $h$, see for example
\cite{MR2573554,MR2091552,ChMe08}.

{For solutions of the Stein equation here} we focus on the general score equation based on \eqref{eq:A},
namely
\begin{equation}
\label{eq:eqstein2}
\Delta u +  \langle \nabla \log p,  \nabla u \rangle = h - \E_p h.
\end{equation}
Bounds on the solution of this Stein equation are available  {(a)} when $p$
is strongly log-concave,  or {(b)} when $p$ admits a Poincar\'e constant. 
{For the first condition,} recall that a smooth density
$p$ is $k$-strongly log-concave for some $k>0$ if
$$\forall (x,y) \in \R^d \times \R^d , \; \left\langle  y-x , \nabla
  \log p(y) - \nabla \log p(x) \right\rangle \;< - k \|y-x\|^2.$$
\cite{mackey2016multivariate} proved that if $p$ is $k$-strongly
log-concave and $h$ is 1-Lipschitz,
then  \eqref{eq:eqstein2} admits
a solution $f=f_h$ such that for $p$-almost every $x$,
\begin{equation}
  \label{prop:mackey}
  |\nabla f_h(x) | \leq \frac{1}{k}.
\end{equation}
Similar bounds based on Malliavin calculus are obtained in
\cite{FSW18}.  
\cite{gorham2016measuring} computed uniform Stein factors for
diffusions satisfying ``distant dissipativity'' and, more generally,
for diffusions that couple rapidly.  This class admits {not only} 
strong log-concave densities {but also} a large class of non-log concave,
multimodal, and heavy-tailed distributions.
 We also point out that the
authors manage to obtain bounds on higher-order derivatives of $f_h$,
under suitable assumptions. {Related results can be found in} 
\cite{FSW18}.

%
%

For the second set of results, if $X \sim p$, recall 
{that $C_p$ is a
  Poincar\'e constant associated to $p$ if for every differentiable
  function $\varphi \in L^2(p)$ such that $\E \varphi(X) = 0$,
$
\E \big[\varphi^2(X)\big] \leq C_p \E \big[|\nabla \varphi(X)|^2\big]
$.}
{Not all densities $p$ admit a finite $C_p$ and probability
distributions which possess a Poincar\'{e} constant are also referred
to as having a spectral gap.}
 Let $ \mathcal C_{c,0}^{\infty}(\R^d) = \{ f \in \mathcal
 C_c^{\infty}(\R^d) \; : \; \int f\, p = 0 \} \subset L^2(p)$. Define
 on $\mathcal C_{c,0}^{\infty}(\R^d)^2$ the Dirichlet form $\mathcal E
 (f,g) = \int  \langle \nabla f, \nabla g\rangle \, p$ and the scalar
 product $\mathcal E_1 (f,g) = \int (f \, g + \langle \nabla f, \nabla
 g\rangle)  p  $. Assume that $\mathcal E$ is {\it closable} (i.e.\
 for all $u_n \in \mathcal 
 C_{c,0}^{\infty}(\R^d); \, n\geq 0$,  if $u_n \rightarrow 0$ in
 $L^2(p)$ and $(u_n)$ is $\mathcal E$-Cauchy, then $\mathcal E
 (u_n,u_n) \rightarrow 0$). This condition ensures that there exists a
 set of functions $W_{1,2}^0(p) \supset \mathcal
 C_{c,0}^{\infty}(\R^d)$, which admit a gradient, and such that
 $W_{1,2}^0(p)$ is a Hilbert space for the scalar product $\mathcal
 E_1$. In particular, by passing to the limit, $\int \phi \, p=0$ for
 all $\phi \in W_{1,2}^0(p)$, and the Poincar\'e inequality holds for
 such $\phi$ if it holds for any differentiable $f \in L^2(p)$ such
 that $\int f\, p = 0$. For background and more material on closable
 Dirichlet forms as well as sufficient conditions on $p$ for a
 Dirichlet form to be closable, we refer to \cite{ma}. 
\begin{prop}
\label{prop:l2bound}
Let $h$ be a  $1$-Lipschitz function. Let $X$ be a random
vector with density $p$ with Poincar\'e constant $C_p$ and satisfying the closability property. Let $W_{1,2}^0(p)$ be defined as above. Then there exists a weak solution $u \in W_{1,2}^0(p)$ to 
\eqref{eq:eqstein2} 
such that
$$\sqrt{\int |\nabla u|^2 p} \leq C_p.$$
\end{prop}

The proof of the result follows exactly the lines of the proof in
\cite{courtade2017existence} and is hence omitted.  We point out that
here $p$ does not need to satisfy Assumption~A.
{
\begin{rem}
For strongly $k$-log concave $p$, \eqref{prop:mackey} gives that there
exists a (strong) solution $u$ such that $|\nabla u (x)| \leq 1/k,$
for all $x$ in the domain. 
Now it is known that when $X$ has $k$-log-concave density, then the
law of $X$ possesses a Poincar\'{e} constant $C_p = 1/k$ see
\cite{BrLi76} (it is not necessary to assume that the density is
$k$-strongly log-concave).
Hence we can also apply  Proposition \ref{prop:l2bound}, which gives  gives
$\sqrt{\int |\nabla u |^2 p} \leq 1/k.$ Thus, the constants in
\eqref{prop:mackey} and in Proposition \eqref{prop:l2bound} are the
same; the bound in Proposition \eqref{prop:l2bound} is weaker only
because the norm is weaker.

\end{rem}

\begin{rem}  
In \cite{courtade2017existence} it is shown that for any mean zero
multivariate distribution $p$ which is absolutely continuous with
respect to the Lebesgue measure with finite second moment, and which
satisfies a Poincar\'e inequality with constant $C_p$, there exists a
unique function $g \in W_{p}^{1,2}$
 such that
$\pmb \tau_p = \nabla g $ is a Stein kernel for $p$. Moreover,
\begin{equation}\label{eq:cofapa}
  \int \| \pmb \tau_p\|_{HS}^2 \,p  \le C_p \int | x|^2 \,p. 
\end{equation}
In the one-dimensional case, \cite{ley2015distances} (inspired by a
similar result from \cite{Do14}) show that there exists a solution $u$
such that \eq
\label{eq:boundlrs}
|u'(x)| \leq \tau_p(x), \qe $\tau_p$ being the univariate Stein kernel
associated to $X$. Under log concavity of the density, in the
univariate case a stronger bound is available in
 \cite{saumard2018weighted}: 
 Assume $X$ is centered and has a smooth, non-vanishing $k$-log
 concave density $p$ on $\R$. Then 
$\tau_p(x)\leq 1/k,$
for all $x\in\R$.
Thus, in the univariate case where $p$ is $k$-log concave, the bound
\eqref{eq:boundlrs} implies \eqref{prop:mackey}, which in turn implies
the $L^2$ bound from Proposition \ref{prop:l2bound}.
\end{rem}

}
\section{{Wasserstein distance between nested distributions}}
\label{sec:applications}

{In this section we illustrate Stein's method in the multivariate setting in three examples which are based on a novel bound on the Wasserstein distance between nested distributions.}
{The} (1-)Wasserstein distance between two
distributions $F$ and $G$ is
$$ d_{\mathcal{W}}(F, G)
= \sup_{h \in {\mathcal{W}}} \left| F(h)- G(h) \right| $$ with
$\mathcal{W} = \mathrm{Lip}(1)$ the collection of Lipschitz functions
$h : \R^d \to \R$ with Lipschitz constant smaller than 1. Here
$F(h) = \int h \,dF$ is the expectation of $h$ under $F$. Abusing
notation, we also write
$ d_{\mathcal{W}} (X, Y) = d_{\mathcal{W}}(F, G)$ when $X\sim F $ and
$Y \sim G$.


{Let} $P_1$ and $P_2$ {be two probability measures} on $\R^d$, with respective pdfs $p_1$
and $p_2 = \pi_0 p_1$, so that the two densities are nested. Assume
$p_1$ is $k$-log concave. Consider as in \eqref{eq:A} a Stein operator
for $P_i, i=1, 2$ defined by
$$ \A_i u  = \langle \nabla \log p_i, \nabla u\rangle   +  \Delta u. $$
Then 
$\A_2 u  = \A_1u+ \langle \nabla \log \pi_0, \nabla u \rangle.$
This relationship between generators makes it straightforward to bound
the Wasserstein distance between the distributions, as follows.   

\begin{prop}
Assume $p_2 = \pi_0 p_1$ and 
that $\E[  |\nabla  \pi_0 (X_1)| ]< \infty$. 
\begin{enumerate}
\item Assume that $p_1$ is   $k$-{\color{red}{strongly}} log concave. Then with $X_1 \sim p_1$ and $X_2 \sim p_2$,
\begin{equation}  \label{prop:nesteddensities} d_{\mathcal{W}}(X_1,X_2) \leq \frac{1}{k} \E[  |\nabla  \pi_0 (X_1)| ]. \end{equation} 
\item 
Assume  that $p_1$ admits a Poincar\'e constant $C_p$ and $\pi_0 \in W_{1,2}(p_1) $. Then with $X_1 \sim p_1$ and $X_2 \sim p_2$,
\eq
\label{eq:boundnested}
d_{\mathcal{W}}(X_1,X_2) \leq  C_p \sqrt{\E[ | \nabla \pi_0(X_1)|^2]}.
\qe
\end{enumerate} 
\end{prop}

\begin{proof} 
Let $h \; :\; \R^d \mapsto \R$ be a $1$-Lipschitz function.
{To show \eqref{prop:nesteddensities}, by   \eqref{prop:mackey}, there exists a solution $u_h $} 
to $\A_1 u_h = h - \int h p_1$ such that $\| \nabla u_h \|_\infty \leq 1/k$. Let $X_1$ ($X_2$) have distribution $P_1$ ($P_2$). Then
\begin{align*}
\E[h(X_2)]-\E[h(X_1)]& = \E[\A_1 u_h(X_2)]\\
&= \E\left[\A_2 u_h(X_2) - \langle  \nabla \log \pi_0(X_2), \nabla u_h(X_2) \rangle  \right]\\
&=- \E\left[  \langle \nabla \log \pi_0(X_2), \nabla u_h(X_2)  \rangle \right].
\end{align*}
By \eqref{prop:mackey} we have $ | \langle \nabla \log \pi_0(X_2), \nabla u_h(X_2) \rangle|  \leq  |\nabla \log \pi_0(X_2)| /k$, and thus
\begin{equation}
\label{eq:boundw}
| \E[h(X_2)]-\E[h(X_1)]| \leq \frac{1}{k} \E[  |\nabla \log \pi_0 (X_2)| ],
\end{equation}
and the same bound holds for the Wasserstein distance. 

To show \eqref{eq:boundnested}, {by  Proposition \ref{prop:l2bound}, there exists a weak solution $u_h  \in W_{1,2}^0(p)$ to 
\eqref{eq:eqstein2} 
such that
$\sqrt{\int |\nabla u|^2 p} \leq C_p.$} 
{It is straightforward to see that 
\begin{equation} \label{lem:ausefulid}
  \EE \left[ \left\langle \nabla u_h(X),  \nabla v(X) \right\rangle \right]
  = \EE \left[ \bar h(X) v(X) \right]
\end{equation}
for any $v \in W_{1, 2}(p_i)$. 
Indeed,} by definition, $  \mathcal{A} f  = \mathcal{T}_{\nabla,
  p}(\nabla f)$  so that,  by \eqref{eq:8},
\begin{equation*}
\EE \left[ \mathcal{A} f(X) g(X) \right]  = \EE \left[ \mathcal{T}_{\nabla,
  p}(\nabla f(X)) g(X) \right]  = - \EE \left[ \left\langle \nabla
  f(X),  \nabla g(X) \right\rangle  \right];
\end{equation*}
and \eqref{lem:ausefulid}follows. 
  Let $\bar h (x) = h(x) - \E[h(X_1)]$. 
  By 
  \ref{lem:ausefulid}, we have for any $v \in W_{1,2}(p)$,
$$\int \langle \nabla u_h, \nabla v \rangle \; p_1 = -\int \bar{h}
\,v\, p_1.$$ Applying this equation to $v = -\pi_0 \in W_{1,2}(p_1)$,
$$\E[ h(X_2) - h(X_1) ] = \int \langle \nabla u_h,  \nabla \pi_0 \rangle  p_1 \leq \left( \int |\nabla u_h|^2 p_1 \int |\nabla \pi_0|^2 p_1 \right)^{1/2} \leq C_p \sqrt{\E[ | \nabla \pi_0(X_1)|^2]}.$$
\end{proof}


\begin{rem}
\label{rem:boundpi}
In dimension 1, in \cite{ley2015distances}, Equation (4.2), {gives} 
\begin{equation}
\label{eq:bounddim1}
d_{\mathcal{W}}(X_1,X_2) \leq \E[\tau_1(X_1) | \pi_0'(X_1) |],
\end{equation}
where $\tau_1$ is the Stein kernel 
associated to $X_1$. If $p_1$ admits a Poincar\'e constant $C_p$, then from  Proposition \ref{prop:l2bound} (applied to $h(x) = -x$), we have that $\int (\tau(x))^2 p_1 \leq C_p^2$. Thus by the Cauchy-Schwarz inequality, \eqref{eq:bounddim1} is a stronger bound than \eqref{eq:boundnested}.
\end{rem}

\subsection{Example 1: Copulas}

Let $(V_1,V_2)$ be a 2-dimensional random vector, such that the marginals $V_1$ and $V_2$ have a uniform distribution on $[0,1]$. We want to bound the Wasserstein distance between $(V_1,V_2)$ and its independent version $(U_1,U_2)$ ($U_1$ and $U_2$ are uniform and independent), in terms of the copula of $(V_1,V_2)$ defined as
$$C(x_1,x_2) = \P[V_1\leq x_1, V_2 \leq x_2], \quad (x_1,x_2)\in [0,1]^2.$$
(Note that the copula for $(U_1,U_2)$ is $(x_1,x_2) \mapsto x_1x_2$.) Assume that $V_1,V_2$ has a {pdf} 
$c = \partial^2_{x_1x_2} C$.
An optimal Poincar\'e constant for the uniform distribution on $[0,1]^2$  is $C_p = 2/\pi^2$, see \cite{payne1960optimal}. Then, a simple application of  \eqref{prop:nesteddensities} yields
\begin{cor} \label{prop:copula} 
Let $(V_1,V_2)$ have uniform marginals on $[0,1]$, and  {pdf} 
 $c$. Let $(U_1,U_2)$ also have uniform marginals, and let $U_1$ be independent of $U_2$. Then
$$d_{\mathcal{W}}[ (V_1,V_2) \; , \; (U_1,U_2)] \leq  \frac{2}{\pi^2} \sqrt{ \int_{[0,1]^2}  |\nabla c(x_1,x_2)|^2 dx_1\, dx_2}.$$
\end{cor}

In some cases, one can compute the gradient of $c$ in a closed form.

\begin{example}
The Ali-Mikhail-Haq copula 
\cite{ali1978class}
{has pdf 
$$ c(x_1, x_2) = \frac{(1-\theta) \{ 1  - \theta(1-x_1)(1 - x_2)\} + 2 \theta x_1 x_2}{\{  1  - \theta(1-x_1)(1 - x_2)\}^3}.$$}
Here $\theta \in [-1,1]$ is a measure of association between the two components $V_1$ and $V_2$ of the vector $(V_1,V_2)$ with uniform marginals each. If $\theta=0$ then the uniform copula $(x_1,x_2) \mapsto x_1x_2$ is recovered. Using 
{Corollary}
\ref{prop:copula} we can assess the Wasserstein distance between the  Ali-Mikhail-Haq copula and the uniform copula in terms of $\theta$. 
for $-1 < \theta < 1$ 
\begin{eqnarray*}
{\int_{[0,1]^2}  |\nabla c(x_1,x_2)|^2 dx_1\, dx_2} 
&\le& 128 \theta^2\frac{1}{\{  1  - |\theta| \}^8}. 
\end{eqnarray*}
Hence 
$$d_{\mathcal{W}} ([(V_1, V_2)\, , \, (U_1, U_2) ]) \le 2.3\, |
\theta|  \{  1  - |\theta| \}^{-4}.$$ 
{This bound shows the expected behaviour -- it} 
tends to 0 for $\theta \rightarrow 0$, whereas it diverges for $| \theta| \rightarrow 1$.
\end{example}

\subsection{Example 2:  Normal model with normal prior}
Consider a normal  ${\mathcal{N}}(\theta , \Sigma_2)$ model with mean $\theta \in \R^d$ and positive
definite covariance matrix $\Sigma$. The likelihood of a sample
$(x_1,\ldots,x_n)$ (where $x_i \in \R^d$ for all $i$) is given by
$\mathcal L(x_1,\ldots,x_n | \theta) = \prod_i f(x_i|\theta)$ where
$$f(x|\theta) = (2\pi)^{-d/2} \det(\Sigma)^{-1/2} \exp\left( -\frac{1}{2}(x-\theta)^T \Sigma^{-1}(x-\theta) \right).$$

To compare the posterior distribution $P_1$ of $\theta$ with
uniform prior with the posterior $P_2$ with normal prior with
parameters $(\mu, \Sigma_2)$
{we employ} the operator
norm $\vvvert A\vvvert = \sup_{\|x\|=1} \|Ax\|.$

\begin{cor}\label{normalpost} Let $P_1$ denote the posterior distribution 
of $\theta$  with uniform prior and $P_2$ the posterior  of $\theta$  with 
{prior ${\mathcal{N}}(\mu, \Sigma_2)$;}
$\Sigma_2$ is assumed positive definite. Then 
  \begin{align*}
    d_{\mathcal{W}} (P_1,P_2) & \leq\vvvert\Sigma\vvvert
                                \;\vvvert(\Sigma+n\Sigma_2)^{-1}\vvvert \;
                                \|\bar x - \mu\|  \\
& \qquad + \frac{\sqrt 2\Gamma(d/2+1/2)}{\Gamma(d/2)} \frac{\vvvert\Sigma\vvvert}{n} \vvvert(\Sigma_2+n\Sigma_2\Sigma^{-1}\Sigma_2)^{-1/2}\vvvert .
  \end{align*}
\end{cor} 

\begin{proof} 
{It is a standard calculation that}
$P_1 \sim \mathcal N(\bar x, n^{-1}\Sigma)$, {with $\bar x = \frac{1}{n}\sum_{i=1}^n x_i$ Thus ({see Example 2.10 in \cite{saumardwellner}}) $p_1$  is {strongly} $1/\lambda$-log concave,where $\lambda$ is the greatest eigenvalue of $\Sigma$; $\lambda$ is also the operator norm of $\Sigma$,}.
{and \eqref{prop:nesteddensities} can be applied}. From \eqref{eq:boundw}, we deduce
$$d_{\mathcal{W}} (P_1,P_2) \leq \frac{\vvvert\Sigma\vvvert}{n} \E[  \|\nabla \log \pi_0 (X_2)\| ]$$
where $X_2$ is a r.v. with law $P_2$. It remains to calculate this
last expectation. 
{It is again a standard calculation that} 
$P_2 \sim \mathcal N(\tilde \mu, \tilde \Sigma_n)$
with
\begin{align}
\tilde \mu &= (\Sigma_2^{-1}+(n^{-1}\Sigma)^{-1})^{-1}( (n^{-1}\Sigma)^{-1}\bar x+\Sigma_2^{-1}\mu)= \mu+n \tilde \Sigma_n \Sigma^{-1} (\bar x - \mu) \label{eq:mu}\\
\tilde \Sigma_n &= (\Sigma_2^{-1}+n\Sigma^{-1})^{-1}. \notag
\end{align}
Since $p_2(\theta) \propto p_1(\theta) \exp[(\theta-\mu)^T \Sigma_2^{-1} (\theta-\mu)]$, it holds that
$$\nabla \log \pi_0(\theta) =  -\Sigma_2^{-1} (\theta-\mu).$$
Together with \eqref{eq:mu}, it follows that $\nabla \log\pi_0(X_2)$
has a normal distribution with mean
$-n\Sigma_2^{-1} \tilde \Sigma_n \Sigma^{-1} (\bar x - \mu)$ and
covariance matrix $\Sigma_2^{-1} \tilde \Sigma_n \Sigma_2^{-1}$. Note
that
\begin{align*}
\Sigma_2^{-1}  \tilde \Sigma_n \Sigma^{-1} &=\Sigma_2^{-1} (\Sigma_2^{-1}+n\Sigma^{-1})^{-1} \Sigma^{-1}
= (\Sigma+n\Sigma_2)^{-1}
\end{align*}
and, in the same way,
$
\Sigma_2^{-1} \tilde \Sigma_n \Sigma_2^{-1} = (\Sigma_2+n\Sigma_2\Sigma^{-1}\Sigma_2)^{-1}.
$
Let $N$ stand for a $d$-dimensional standard normal vector. From the above, we deduce that $\nabla\log \pi_0(X_2) \sim -n(\Sigma+n\Sigma_2)^{-1} (\bar x - \mu) +  (\Sigma_2+n\Sigma_2\Sigma^{-1}\Sigma_2)^{-1/2} N$. Thus,
\begin{align*}
\frac{\vvvert\Sigma\vvvert}{n} \E[  \|\nabla \log \pi_0 (X_2)\| ] \leq &\vvvert\Sigma\vvvert \;\vvvert(\Sigma+n\Sigma_2)^{-1}\vvvert \; \|\bar x - \mu\|\\
&  + \frac{\vvvert\Sigma\vvvert}{n} \vvvert(\Sigma_2+n\Sigma_2\Sigma^{-1}\Sigma_2)^{-1/2}\vvvert \E[\|N\|].
\end{align*}
Since $\|N\|^2 \sim \chi^2(d)$, $\|N\|$ follows a chi distribution
with $d$ degrees of freedom and
$\E[\|N\|] =\frac{\sqrt 2\Gamma(d/2+1/2)}{\Gamma(d/2)}$ .
The assertion follows.
\end{proof}

\begin{remark}
  The bound in Corollary \eqref{normalpost} has a contribution of
  order $n^{-1}$ from the different covariance matrices, but also a
  contribution which depends on $ \|\bar x - \mu\|$. By the law of
  large numbers, this term is order $n^{-1}$ in probability, but
  crucially depends on the actual sample which was drawn.
\end{remark}

\begin{remark}
When $d=1$, we retrieve the bound of \cite{ley2015distances}. 
\end{remark}

 As in \cite{ley2015distances}, our Stein framework lends itself naturally for assessing distributional distances between nested densities more generally than in the previous example. 

\begin{prop} \label{sec:nested-densities}
Assume $p_2 = p_1 \pi_0$, where $p_1$ and $p_2$ are densities which satisfy Assumption A. Let $X_1 \sim p_1$ and $X_2 \sim p_2$ denote two random variables with distributions having densities $p_1$ and $p_2$, respectively. Then
\begin{equation} \label{nested} \|\E[ \pmb \tau_1(X_1) \nabla
  \pi_0(X_1) ]\|\leq d_{\mathcal{W}} (X_1,X_2) \leq \sup_{f_h} \, | \E
  [ \langle\nabla\log \pi_0 (X_2), \nabla f_h (X_2) \rangle)] |
\end{equation}
where the supremum is taken over all $f_h$ solving the $p_1$-Stein
equation \eqref{eq:eqstein2} for $h$ a 1-Lipschitz function.
\end{prop} 

\begin{proof} 
Consider the lower bound first. 
Let $e$ be a unit vector. Since $x \mapsto \langle x,e\rangle$ is 1-Lipschitz, using the nested structure we have
\begin{align*}
d_{\mathcal{W}} (X_1,X_2)& \geq \E [ \langle X_1,e\rangle - \langle X_2,e\rangle]\\
& = \E[ \langle X_1,e \rangle (1-\pi_0(X_1)) ]\\
& = -\langle \E[  \pmb \tau_1(X_1) \nabla \pi_0(X_1)  ] , e \rangle.
\end{align*}
Taking $e = - \E[  \pmb \tau_1(X_1) \nabla (p_2/p_1) (X_1)  ] / \|\E[  \pmb \tau_1(X_1) \nabla \pi_0(X_1)  ] \|$ gives 
$$d_{\mathcal{W}} (X_1,X_2) \geq \|\E[  \pmb \tau_1(X_1) \nabla \pi_0(X_1)  ]\|.$$
For the upper bound,  if $h$ is 1-Lipschitz, using the score Stein operator for $p_1$, 
\begin{align*}
\E[h(X_2) - h(X_1)] & = \E [(\Delta f_h )(X_2) - \langle\nabla\log  p_1 (X_2), \nabla f_h (X_2) \rangle)]\\
&= \E [ \langle\nabla\log  \pi_0 (X_2), \nabla f_h (X_2) \rangle)]. 
\end{align*}
The last equality follows from the nested structure $p_2 = p_1 \theta_0$. 
\end{proof}

\begin{remark}\label{likelihhood ratio}
  The following argument shows that the gradient of the likelihood
  ratio between two densities arises naturally in the Stein
  framework. Suppose that $q$ is another density on $\R^d$ and that
  $K_p$, the support of $p$, is a subset of $K_q$, the support of
  $q$. Then the likelihood ratio $r=p/q$ is well defined over $\R^d$
  (with the convention that $r=0$ outside of $K_q$) and, for every
  $\mathbf{f} \in \mathcal{F}(p) \cap \mathcal{F}(q)$, the product
  rule \eqref{productrule} gives
\begin{align}
  \mathcal{T}_{\dv, p} \mathbf{f}& =
                            \frac{\mathrm{div}(\mathbf{f}\,p)}{p}
 = \frac{\mathrm{div}(\mathbf{f}\,q\, r)}{q} \frac{1}{\ell}
                            \nonumber \\
  &= \frac{\mathrm{div}(\mathbf{f}\,q)}{q}  +
    \big\langle \mathbf{f},  \frac{\nabla r}{r} \big\rangle  =
    \mathcal{T}_{\dv, q} (\mathbf{f}) + \big\langle
    \mathbf{f} , 
    \frac{\nabla r}{r} \big\rangle. 
    \label{eq:2}
\end{align}
In combination with the ideas of Proposition
\ref{sec:nested-densities}, this leads to a general bound on the
difference between densities in terms of the likelihood ratio. 
\end{remark}

\begin{remark} 
\label{sec:general-prior}
Based on Remark \ref{likelihhood ratio}, the comparisons of the
posteriors arising from different priors can be carried out more
generally. Here is a sketch.  In the Bayesian framework, we aim to
compare $\Theta_1$ and $\Theta_2$ obtained through different priors,
where for $\ell_{\theta}(x)$ the likelihood function,
\begin{align*}
  & \Theta_1 \sim \pi_1(\theta) = \kappa_1(x) \ell_{\theta}(x) \\
 &   \Theta_2 \sim \pi_2(\theta) = \kappa_2(x) \pi_0(\theta) \ell_{\theta}(x)
\end{align*}
for $\pi_0$ a nonnegative function with support a subset of $\R^d$,
and $\kappa_i(x)$, $i=1, 2$ the normalizing constants -- here
$x = (x_1, \ldots, x_J) \in \R^J$ is a fixed sample of size $J$ and
$\theta \in \R^d$ is the variable.  Then, in the notations of
Remark~\ref{likelihhood ratio}, we have 
$r(\theta) = \kappa_2(x) / \kappa_1(x) \pi_0(\theta)$ so that,
introducing $\rho_0 = \nabla \pi_0/\pi_0$ -- which does not depend on
the normalizing constants -- we obtain, from \eqref{eq:2}, 
\begin{equation}\label{eq:38}
  \mathcal{T}_{\dv, \pi_2} \mathbf{f} =
  \mathcal{T}_{\dv, \pi_1}\mathbf{f} + 
  {\langle \mathbf{f},  \rho_0 \rangle}
\end{equation}
for any function
$\mathbf{f} \in \R^n \times \R^d \in {\mathcal{F}}(\pi_1) \cap
{\mathcal{F}}(\pi_2)$, where $n \ge 1$ is arbitrary.  In particular,
\begin{equation}
  \label{eq:73}
  \E \left[ \mathcal{T}_{\mathrm{div}, \pi_1} \mathbf{f}(\Theta_2)
  \right] = \E \left[ \left\langle \mathbf{f}(\Theta_2), \rho(\Theta_2) \right\rangle \right]
\end{equation}
for all sufficiently regular vector-valued functions $\mathbf{f}$.
Next suppose that there exists some well chosen matrix valued function
$\mathbf{B}$ for which $\pi_1$ is characterized by a second-order
divergence Stein operator acting on real valued functions
$u : \R^d \to \R$ via
\begin{equation}
  \label{eq:71}
u \mapsto  \mathcal{A}_1u = \mathcal{T}_{\dv, \pi_1} (\mathbf{B}
\nabla u) = \left\langle \mathbf{b},\nabla u \right\rangle +
\left\langle  \mathbf{B},\nabla^2u \right\rangle_{\mathrm{HS}}
\end{equation}
with $\mathbf{b} = \mathcal{T}_{\dv, \pi_1} \mathbf{B}$ (recall
\eqref{eq:axfnuhess2}); then
in light of \eqref{eq:38}, we know that  
\begin{align}\label{eq:14}
  \mathcal{A}_2u =   \mathcal{T}_{\dv, \pi_2} (\mathbf{B}
\nabla u) = \mathcal{A}_1 u +
  \langle \mathbf{B}\nabla u, \rho_0 \rangle
\end{align}
satisfies $\E \big[ \mathcal{A}_2u(\Theta_2)\big]=0$ for all
admissible functions $u$. In particular, if $u_h$ is a solution to the
(second order) Stein equation 
\begin{align*}
  \mathcal{A}_1 u_h(\theta) =  h(\theta) - \E \left[ h(\Theta_1) \right]
\end{align*}
we get, with the help of \eqref{eq:14} and under {{suitable}}
conditions,
\begin{align*}
  \E \left[ h(\Theta_2) \right] -\E \left[ h(\Theta_1) \right] & =\E
                                                                 \left[
                                                                 \mathcal{A}_1
                                                                 u_h(\Theta_2)                                                      
                                                                 \right] \\
                                                               & =\E
                                                                 \left[
                                                                 \mathcal{A}_2u_h(\Theta_2) 
                                                                 \right]
                                                                 -
                                                                 \E
                                                                 \left[                                                                                             \langle
                                                                 \mathbf{B}(\Theta_2)\nabla
                                                                 u_h(\Theta_2),
                                                                 \rho_0
                                                                 (\Theta_2)
                                                                 \rangle                                                                    
                                                                \right] \\
  & = -  \E \left[  \langle \mathbf{B}(\Theta_2) \nabla u_h(\Theta_2),
    \rho_0 (\Theta_2) \rangle     \right]. 
\end{align*}
Investigating this approach in more detail will be part of future
research.
\end{remark} 

\subsection{{Example 3: Comparing Azzalini-Dalla Valle skew-normal distributions to multivariate normal }}

The density of the Azzalini-{Dalla Valle} type r.v. $X$ is given by
$$p_\alpha(x)=2\omega_d(x-\mu;\Sigma) \Phi(\alpha^T x),$$
where $\omega_d(x-\mu:\Sigma)$ is the density of
$\mathcal N(\mu,\Sigma)$, $\Phi$ the c.d.f. of the standard normal on
$\R$, and $\alpha \in \R^d$ is a skew parameter, {see
  \cite{azzalini1996}}. Now we assume for simplicity $\Sigma = I_d$
and $\mu=0$. {In \cite{LRS16}, an exact expression for the
  Wasserstein distance is given for $d=1$; here we extend this result
  to general $d$.}

\begin{prop}
Let $X \in \R^d$ have pdf $p_\alpha(x)=2\omega_d(x; Id ) \Phi(\alpha^T x),$, and let $Z$  be a $d$-dimensional standard normal element. Then 
\begin{equation}\label{wassequality} 
d_{\mathcal W}(X,Z) = \sqrt{\frac{2}{\pi}}\frac{\|\alpha\|}{\sqrt{1+\|\alpha\|^2}}.
\end{equation}
\end{prop}
\begin{proof}
{First we show the upper bound on the Wasserstein distance. As the multivariate standard normal distribution is 1-strongly log concave,\eqref{prop:nesteddensities}  can be applied with
$\pi_0 (x) = 2 \Phi (\alpha^T x). $
As $\nabla \pi_0 (x) = 2\alpha \phi (\alpha^T x)$
with $\phi$ the one-dimensional standard normal pdf, it suffices to calculate that}
\begin{align*}
 \E[\phi(\alpha^T Z)] &= (2\pi)^{-(d+1)/2}\int_{\R^d} \exp\left[ - \frac{1}{2}\left( (\alpha^T x)^2 + \|x\|^2 \right) \right]dx\\
 &=(2\pi)^{-(d+1)/2}\; \frac{(2\pi)^{d/2}}{\sqrt{\text{det}(M)}}\\
&=\frac{1}{\sqrt{2\pi(1+\|\alpha\|^2)}}.
\end{align*}
Thus,
\begin{equation}
d_{\mathcal W}(X,Z) \leq  \sqrt{\frac{2}{\pi}}\frac{\|\alpha\|}{\sqrt{1+\|\alpha\|^2}}.
\end{equation}
Now we prove we have actually equality. Consider the $1$-Lipschitz
test function $h(x) = \langle \frac{\alpha}{\|\alpha\|}, x \rangle$;
note that $\E[h(Z)]=0$. Then 
$u_h(x) = -h(x)$
is a solution to the Stein equation
$\Delta u - \langle x, \nabla u(x) \rangle= h(x) $
and $\Delta u=0$, {and} 
\begin{align*}
  \E[h(X)-h(Z)] &= 2  \E [\Delta u_h(X) - \langle X,\nabla u_h(X) \rangle ] \\
  &=  {2\E[ \langle  \nabla u_h(Z), \alpha  \phi (\alpha^T Z) \rangle]}\\
                &= 2\E[\langle \nabla h(Z),  \alpha  \phi (\alpha^T Z) \rangle]\\
                &= 2\E \left[ \frac{1}{\|\alpha\| }  \langle \alpha, \alpha  \phi(\alpha^T Z) \rangle \right]\\
&=\sqrt{\frac{2}{\pi}}\frac{\|\alpha\|}{\sqrt{1+\|\alpha\|^2}},
\end{align*}
so that $d_{\mathcal W}(X,Z) \geq  \sqrt{\frac{2}{\pi}}\frac{\|\alpha\|}{\sqrt{1+\|\alpha\|^2}}$. Thus we obtained \eqref{wassequality}.
\end{proof}

\section{Stein operators for elliptical distributions}
\label{sec:ellipt-distr}

In this section we detail the constructions explicitly for the entire
family of elliptical distributions, {defined as follows, see}  
\cite{landsman2008stein}.
\begin{definition}
An absolutely continuous $d$-random vector has multivariate elliptical
distribution $\E_d(\nu, \Sigma, \phi)$ if its density is of the form
\begin{equation}
  \label{eq:ellipticdistr}
  p(x) = \kappa|\Sigma|^{-1/2} \phi \left( \frac{1}{2}
    (x-\nu)^T\Sigma^{-1}(x-\nu) \right), \, x \in \R^d,
\end{equation}
for $\phi : \R^+ \to \R^+$ a measurable function called \emph{density
  generator}, $\nu \in \R^d$ the location parameter, $\kappa$ the
normalising constant and $\Sigma = (\sigma_{ij})$ a symmetric positive definite
$d \times d$ dispersion matrix.
\end{definition}

   A particular important case is $\E_d(0, I_d, \phi)$ called
   \emph{spherical distribution}. Note that the matrix $\Sigma$ in
   definition \eqref{eq:ellipticdistr} is not necessarily the
   covariance matrix; also not all choices of $\phi$ lead to
   well-defined densities, see \cite{landsman2008stein} or
   \cite{pain14} for a discussion and references.  Prominent members
   of the elliptical family are
\begin{enumerate}
\item Gaussian distribution $\mathcal{N}_d({\nu},  \Sigma)$,  with $\phi(t) = e^{-t}$ and $\kappa =
  (2\pi)^{-d/2}$. 
\item Power exponential distribution 
  $\phi(t) = \mathrm{exp}(- b_{p, \zeta} t^{\zeta})$ with $\zeta>0$,
  $ b_{p, \zeta}$ a scale factor and $\kappa$ defined accordingly, see
  \cite{gomez2008,aerts2017robust} for details.
\item Multivariate Student-$t$ distribution,
with
  $\phi(t) = \left( 1+2t/k \right)^{-(k+d)/2}$ and
  $\kappa = c_{k, d, \Sigma}$.
\item Symmetric generalized hyperbolic distribution with density 
  \begin{equation}
    \label{eq:66}
    p(x) =  \frac{\left( \sqrt{\chi \psi} \right)^{-\lambda}
      \psi^{d/2}}{(2\pi)^{d/2} \left| \Sigma \right| K_{\lambda}
      \left( \sqrt{\chi \psi} \right)} \frac{K_{\lambda - d/2}\left(
        \sqrt{\left( \chi + (x-\nu)^T \Sigma^{-1} (x - \nu)
          \right)\psi} \right)}{\left( \sqrt{\left( \chi + (x-\nu)^T \Sigma^{-1} (x - \nu)
          \right)\psi} \right)^{(d/2)-\lambda}}
  \end{equation}
where $K_{\lambda}$ denotes a modified Bessel function of the third
kind and $\lambda, \chi, \psi$ are real parameters, see
  \cite[Example 3.8]{mcneil2015quantitative} for details. To put in
  the parameterization \eqref{eq:ellipticdistr} we take $\phi(t) =
  K_{\lambda - d/2}(\sqrt{2t})/(\sqrt{2t})^{(d/2)-\lambda}$ and 
  $\kappa = \frac{\left( \sqrt{\chi \psi} \right)^{-\lambda}
      \psi^{d/2}}{(2\pi)^{d/2} \left( \sqrt{\chi \psi} \right)}$.
\end{enumerate}

Throughout this section 
we let
$\Omega_\phi$ be the image of $\Omega_p$ through
$x\mapsto\frac{1}{2} (x-\nu)^T\Sigma^{-1}(x-\nu)$. 
{We make the following assumption.

\medskip 
\noindent \textbf{Assumption B}:
There is a non-empty
  open set $\Omega_\phi$ such that $\phi>0$ and $\phi$ is $\mathcal
  C^1$ on this open set, and such that the Lebesgue measure of
  $K_\phi\backslash \Omega_\phi$ is zero.
 
 \medskip 
Under Assumption B, it is readily checked
  that $p$ satisfies Assumption A with $\Omega_p$ defined by \eq
\label{eq:omegapell}
\begin{array}{lll}
\Omega_p &= &\left\{ x \in \R^d \; : \;  \frac{1}{2}
    (x-\nu)^T\Sigma^{-1}(x-\nu) \in \Omega_\phi\right\}\\
    &=&\left\{ x \in \R^d \; : \;  \phi \left( \frac{1}{2}
    (x-\nu)^T\Sigma^{-1}(x-\nu) \right) >0\right\}.
    \end{array}
\qe}

%

{It is straightforward to calculate that if} $p$ is of the form \eqref{eq:ellipticdistr} {and satisfies Assumption B} then
\begin{equation}
  \label{eq:elliptscore}
  \rho_p(x) = \Sigma^{-1}(x-\nu)
  \frac{\phi'((x-\nu)^T\Sigma^{-1}(x-\nu)/2)}{\phi((x-\nu)^T\Sigma^{-1}(x-\nu)/2)}
\end{equation}
is the score function of $p$ (defined on $\Omega_p$).  Hence the
score-Stein operator \eqref{eq:1} is easily obtained for this family
of distributions.

\subsection{Stein kernels for elliptical distributions} 

The following 
proposition shows that in order to find Stein
kernels for elliptical distributions, it suffices to consider the case
$\Sigma = I_d$ and $\nu=0$.
 \begin{prop}
 \label{prop:sigmaid}
 The application
 $$ \pmb \tau \mapsto (x\mapsto \Sigma^{1/2} \pmb \tau(\Sigma^{-1/2} (x-\nu)) \Sigma^{1/2}),$$
 is a bijection between the set of Stein kernels of $\E_d(0, I_d, \phi)$ and that of $\E_d(\nu, \Sigma, \phi)$.
 \end{prop}
 \begin{proof}
Let $p$ (resp. $q$) be the density of $\E_d(0, I_d, \phi)$ (resp. $\E_d(\nu, \Sigma, \phi)$). If $\pmb \tau_X$ is a Stein kernel for $X \sim \E_d(0, Id, \phi)$, then for all $f \in \mathcal C^\infty_c(\Omega_p)$, $\E[\pmb \tau_X(X) \nabla f(X)] = \E[ Xf(X)]$. Setting $f(x) = g(\Sigma^{1/2}\, x+\nu)$, we deduce that for all $g\in \mathcal C^\infty_c(\Omega_q)$, we have $\E[\pmb \tau_X(X) \Sigma^{1/2}\nabla g(\Sigma^{1/2}X+\nu)] = \E[X g(\Sigma^{1/2}X+\nu)]$. Thus
$$\E[\Sigma^{1/2}\pmb \tau_X(\Sigma^{-1/2}( (\Sigma^{1/2}X+\nu) - \nu)) \Sigma^{1/2}\nabla g(\Sigma^{1/2}X+\nu)] = \E[(\Sigma^{1/2}X+\nu-\nu) g(\Sigma^{1/2}X+\nu)].$$
It follows $x\mapsto \Sigma^{1/2}\pmb \tau_X(\Sigma^{-1/2}(x - \nu)) \Sigma^{1/2}$ is a Stein kernel for $\Sigma^{1/2}X+\nu$. The converse is shown in the same way.
 \end{proof}

 Now we state the following result, due to
 \cite{landsman2008stein,landsman2015some} but for which we give a new
 proof.
\begin{lemma}[Proposition
2, \cite{landsman2015some}]\label{lemm:lemmasteik}
If $X \sim E_d(\nu, \Sigma, \phi)$ then the matrix
\begin{equation}
  \label{eq:elliptstek}
\pmb \tau_1(x)  = \left( \frac{1}{\phi((x-\nu)^T\Sigma^{-1}(x-\nu)/2)}
\int_{(x-\nu)^T\Sigma^{-1}(x-\nu)/2}^{+\infty}\phi(u) du \right) \Sigma,
\end{equation}
is a Stein kernel for $X$ if $\pmb \tau_1 \in \mathcal{F}(p)$. 
\end{lemma}
\begin{proof} For transparency here we consider general $\nu$ and
  $\Sigma$.
  We will show the Stein identity holds for functions in
  $\mathcal C^\infty_c(\Omega_p)$. Consider any functions
  $f : \R \to \Omega_\phi \, \in \mathcal{F}(p)$ and
  $g : \Omega_p \to \R \, \in \mathrm{dom}(p,f)$.  We
  start by inverting \eqref{eq:elliptscore} to get
 \begin{equation*} 
   x-\nu = \Sigma\rho_p(x) 
  \frac{\phi((x-\nu)^T\Sigma^{-1}(x-\nu)/2)}{\phi'((x-\nu)^T\Sigma^{-1}(x-\nu)/2)}. 
 \end{equation*}
 Fixing $\nu=0$ (see Proposition \ref{prop:sigmaid}) and introducing the temporary notations
 $\psi(t) = \phi(t)/\phi'(t)$ and $t = x^T\Sigma^{-1}x/2$,
$T = X^T\Sigma^{-1}X/2$, by 
 \eqref{eq:3}, 
\begin{align*}
&  \E_p \left[ Xf(T)g(X) \right]   = \E_p \left[ \Sigma\rho_p(X) 
  \psi(T)
 f(T) g(X) \right] \\
& = - \Sigma   \E_p \left[ \nabla \left\{
  \psi(T)f(T)
  g(X) \right\} \right] \\
& = - \Sigma   \E_p \left[ \nabla \left\{
  \psi(T)f(T)
 \right\}  g(X)  \right] - \Sigma  \E_p \left[
  \psi(T)f(T)
  \nabla g(X)\right] \\
&  = - \Sigma   \E_p \left[   \left(   \psi'( T)f(T)+  \psi( T)f'(T)\right)
\Sigma^{-1}X g(X)  \right]  - \Sigma  \E_p \left[
  \psi(T) f(T)
  \nabla g(X)\right]
\end{align*} 
and thus 
\begin{align}\label{eq:70}
 \E_p \left[ \left( f(T)(1+  \psi'( T))+  \psi( T)f'(T)
  \right)  X g(X)\right] = - \E_p \left[\psi(T)f(T) \Sigma  \nabla
  g(X)\right]. 
\end{align}
In order to obtain a Stein kernel, it suffices to choose $f$ solution
to the ODE
\begin{align*}
  f(t)(1+  \psi'( t))+  \psi(t)f'(t) = -1 
\end{align*}
to ensure that  the function $x \mapsto \psi(t)f(t) \Sigma $ satisfies
\eqref{eq:defstek2},  and is a Stein kernel in the sense of Definition
\ref{def:multiteikdef}. 
Now note that the function
$u(t) := \frac{1}{\phi(t)} \int_{t}^{+\infty}\phi(u) du$ satisfies
$u' = - \frac{1}{\psi}  u - 1$; hence the choice $f=u/\psi$ satisfies
$(f\psi)' = (u)'=  - \frac{1}{\psi}  u - 1$  and thus 
is exactly what we need. Setting $t = x^T\Sigma^{-1}x/2$ 
the claim follows.
\end{proof}
\begin{rem}
In dimension $d=1$, the Stein kernel \eqref{steinkernel1d} of $p$ is  the function
    $ \tau_p(x) = \frac{1}{p(x)} \int_x^{+\infty} (u-\nu) p(u) du$. 
     Changing
    variables in \eqref{eq:elliptstek} for $x \ge 0$,
\begin{align*}
 \pmb \tau_p(x) &  =  \frac{1}{\phi((x-\nu)^2/2)}
\int_{(x-\nu)^2/2}^{+\infty}\phi(u) du \\ 
  & = \frac{1}{p(x)}
    \int_{(x-\nu)^2/2}^{+\infty}p(\sqrt{2u}) du 
   = \frac{1}{p(x)} \int_x^{\infty} (y-\nu) p(y) dy.
\end{align*}
The case $x<0$ is treated similarly. Hence \eqref{eq:elliptstek}
indeed recovers the Stein kernel.
\end{rem}

  The identity \eqref{eq:elliptstek} {resulting from  Lemma~\ref{lemm:lemmasteik}} has
  {found many applications, \cite{adcock2007,adcock2010,landsman2013note,adcock2014,landsman2015some,VY2017}
  and the references therein.}

The following proposition gives a way of finding Stein kernels of a
particular form which generalizes Lemma \ref{lemm:lemmasteik}.

\begin{prop}
\label{prop:stkernelell}
Let $a, b \; : \; \Omega_\phi \rightarrow \R$ two $\mathcal C^1$
functions such that, for all all $u\in \Omega_\phi$,
\begin{equation}
\label{eq:linkab}
\frac{(a(t) \phi(t))'}{\phi(t)}   + 2t \frac{(b(t) \phi(t))'}{\phi(t)}
+ (d+1)b(t) + 1 = 0   
\end{equation}
and set
$$\pmb \tau_{a, b}(t) = a(t) \Sigma + b(t) \,(x-\nu)(x-\nu)^T$$
with $t = \frac{1}{2}(x-\nu)^T \Sigma^{-1} (x-\nu)$. If
$\pmb \tau_{a, b} \in \mathcal{F}(p)$ then this function is a Stein
kernel for $X \sim \E_d(\nu, \Sigma, \phi)$.  Moreover,
  if $\phi$ is continuous and positive on $[0,+\infty)$, then the
  Stein identity \eqref{eq:steinkernelidentity} holds for every test
  function $g \in \mathcal C^1(\R^d)$ such that $g(x)a(t)\phi(t)$ and
  $g(x)b(t) t$ go to zero when $|x|\rightarrow+\infty$.
\end{prop}
\begin{proof}
From Proposition \ref{prop:sigmaid}, we can assume $\nu=0$ and
$\Sigma= I_d$. It is readily checked that
$
\dv( I_d ) = 0, \, 
\dv(x x^T) = (d+1) x,  \,  x x^T x = 2t x.
$
Thus, by the chain rule and noting that $\nabla t = x$,
\begin{align*}
&\dv(\phi(t) (a(t) I_d + b(t) x x^T) ) \\
  &= (a(t) \phi(t))' I_d \, x + a(t)\phi(t) \dv(I_d) + (b(t) \phi(t))' x x^T \, x + b(t)\phi(t) \dv(x x^T) \\
&=  (a(t) \phi(t))' x  + 2t (b(t) \phi(t))' x + (d+1)b(t)\phi(t) x.
\end{align*}
Hence $a(t) I_d + b(t) x x^T$ is a Stein kernel if the last quantity is equal to $-\phi(t) x$, so that  the result follows.

To see that a test function satisfying the stated conditions verifies
\eqref{eq:steinkernelidentity}, simply note that since
$(x-\nu)(x-\nu)^T$ is of order $t$ when $|x|$ is large, the conditions
imply that $\pmb\tau(x) p(x) f(x) \rightarrow 0$ when $|x|$ is large,
and the result follows from the divergence theorem (see discussion
below Definition \ref{def:multiteikdef}). 
\end{proof}

\begin{remark}
\begin{enumerate} 
\item A particular instance of Proposition \ref{prop:stkernelell} is
  given by the following expression: 
\begin{equation}\label{lem:steikellipt2}
   \pmb \tau_{2, \beta}  (x) =   \frac{(\beta+ 2)-2 \frac{\phi''(t)/\phi'(t)}{\phi'(t)/\phi(t)}}{(\beta-2)(d-1)}\left(  2 \left(
      \frac{d-1}{(\beta+ 2) \frac{\phi'(t)}{\phi(t)} - 2
      \frac{\phi''(t)}{\phi'(t)}} + t \right)\Sigma - (x-\nu) (x-\nu)^T \right) \end{equation}
is a Stein kernel for $X$   for all $\beta \neq 2$ as
    long as $\pmb \tau_{2, \beta} \in \mathcal{F}(p)$. 
\item 
  It is straightforward to generalize the previous proposition in the
  following way. Here without loss of generality we take
  $\Sigma = I_d$ and $\nu=0$. Assume we are given matrices
  $\mathbf U_1, \ldots, \mathbf U_m$ such that for every
  $i = 1,\ldots,m$ and some functions
  $\alpha_i, \beta_i \; : \; \R \rightarrow \R$,
$$\dv \mathbf U_i = \alpha_i(t) x, \quad \mathbf U_i x = \beta_i(t) x.$$
If
$$1+ \sum_{i=1}^m a_i \alpha_i + \frac{(a_i \phi)'}{\phi} \beta_i = 0,$$
then $a_1(t) \mathbf U_1(t) + \ldots + a_m(t) \mathbf U_m(t)$ is a
Stein kernel for $X \sim \E_d(0, I_d, \phi)$ if this function is in
the class $ \mathcal{F}(p)$.
\end{enumerate} 
\end{remark}

By setting $b \equiv 0$, we obtain
$a(t) =\frac{1}{\phi(t)} \int_t^{+\infty} \phi(u) du,$
and we retrieve  Lemma \ref{lemm:lemmasteik}.
Setting $a \equiv 0$ leads to the following
\begin{cor} Set  $t = \frac{1}{2}(x-\nu)^T \Sigma^{-1} (x-\nu)$. If $ \int_0^{+\infty} u^{\frac{d-1}{2}} \phi(u) du < \infty$, then
\begin{equation}
\label{eq:stkerellip2}
\left( \frac{t^{-\frac{d+1}{2}}}{2\phi(t)} \int_t^{+\infty} u^{\frac{d-1}{2}} \phi(u) du \right)(x-\nu)(x-\nu)^T,
\end{equation}
 is a Stein kernel for  $X \sim \E_d(\nu, \Sigma,
 \phi)$.  Moreover, if $\Omega_p = \R^d$, the Stein
   identity \eqref{eq:steinkernelidentity} holds for every function $f
   \in \mathcal C^1(\R^d)$ such that $f(x) t^{\frac{d-1}{2}}$ is
   bounded. 
\end{cor}
\begin{proof}
With $a\equiv 0$, \eqref{eq:linkab} becomes
$$(b\phi)' + \frac{d+1}{2t} b\phi = - \frac{\phi}{2t},$$
which integrates to
$$b(t) = \frac{t^{-\frac{d+1}{2}}}{2\phi(t)} \int_t^{+\infty} u^{\frac{d-1}{2}} \phi(u) du.$$

Now if $f$ is as stated in the corollary, since $(x-\nu)(x-\nu)^T$ is
of order $t$ for large $|x|$, then
$\tau(x) f(x) p(x) = \mathcal O(\int_t^{+\infty} u^{\frac{d-1}{2}}
\phi(u) du )$ for large $|x|$, so that $\pmb\tau(x) f(x) p(x)$ goes to
zero at infinity, and the Stein identity follows again from the
divergence theorem.
\end{proof}
\begin{rem}
For $d=1$, \eqref{eq:stkerellip2} leads to the classical Stein kernel \eqref{steinkernel1d}. Indeed, assuming $\nu=0$, for $x>0$,
\begin{align*}
\frac{t^{-1}}{2\phi(t)} \int_t^{+\infty}  \phi(u) du & = \frac{1}{x^2p(x)} \int_{x^2/2}^{+\infty} p(2\sqrt u) du \\
&= \frac{1}{x^2p(x)} \int_x^{+\infty} s\, p(s) ds,
\end{align*}
and multiplying by $x^2$ yields the claim. The case $x< 0$ is treated similarly.
\end{rem}

In the next three subsections we develop Stein kernels
    for three distributional families: the multivariate Gaussian, the
    power exponential, and the multivariate Student $t$-distribution.
    Similar computations are possible for the symmetric generalized
    hyperbolic distribution but are not pursued here.
    We refer to \cite{landsman2015some,VY2017} and the references
    therein.

 \subsection{The multivariate Gaussian
   distribution}
\label{sec:stein-oper-gauss}
Consider a Gaussian $d$-dimensional random vector
$Z \sim \mathcal{N}_d({\nu},  \Sigma)$ with pdf $\varphi$ 
on $\R^d$ and let $\mu(dx) = \varphi(x) dx$ be the corresponding
probability measure. {As}
 $Z \sim E_d (\nu, \Sigma,\phi)$ with
$\phi(t) = e^{-t}$ {and}  $\phi'(t)/\phi(t) = -1$, from
\eqref{eq:elliptscore}, we {recover that}
$ \rho_{\varphi}(x) = - \Sigma^{-1} (x-\nu)$ is the score function of
$\varphi$. Since $\frac{1}{\phi(t)} \int_t^{\infty} \phi(u) du = 1$
for all $t$, Lemma \ref{lemm:lemmasteik} shows that
$ \pmb \tau_{1} = \Sigma$ is, as is well-known, a Stein kernel for
$\varphi$. 
Moreover, 
\eqref{lem:steikellipt2} gives, after some
simplifications, the following family of Stein kernels which are indexed by
$\beta \neq 2$ (we set $\nu=0$ to save space): 
    \begin{equation*}
      \pmb \tau_{2, \beta}(x) = 
\frac{\beta}{(\beta-2)(d-1)} \big(x^T\Sigma^{-1}x\Sigma -x x^T \big) -
\frac{2}{\beta-2}  \Sigma.  
    \end{equation*}
It is easy to check that  these functions are in the class $ \mathcal{F}(p)$.
Several interesting cases stand out. Sending
$\beta$ to 0, on the one hand, and to $+\infty$ on the other hand, we
obtain
\begin{align*}
\pmb \tau_{2, 0}(x) = \Sigma \mbox{ and }  \pmb \tau_{2, \infty}(x) =
  \frac{1}{d-1} \big(x^T\Sigma^{-1}x\Sigma- x x^T\big).
\end{align*}
In dimension $d \ge 3$, setting $\beta = 2(d-1)$ we get
    \begin{equation*}
     \pmb \tau_{2, 2(d-1)}(x) = \frac{1}{d-2} 
 \big(x^T\Sigma^{-1}x\Sigma- x x^T  - \Sigma\big)
    \end{equation*}
    {Moreover, we find} that the Gaussian
    multivariate normal satisfies
\begin{equation}
  \label{eq:genstekcovid}
\frac{\beta}{(\beta-2)(d-1)}  \EE \left[\big(X^T\Sigma^{-1}X\Sigma -X
  X^T \big) \nabla g(X) \right]  = \EE \left[ \frac{2}{\beta-2}
  \Sigma \nabla  g(X)  + X g(X)\right]
\end{equation}
for all $g \in \mathrm{dom}(\varphi, \pmb \tau_{2, \beta})$ and all
$\beta \neq 2$.  In particular, every $g \in \mathcal C^1(\R^d)$ with
at most polynomial growth at infinity lies in this domain.

\subsection{Power exponential distribution}
\label{sec:power-expon-distr}

Consider a $d$ random vector
$Z \sim \mathrm{PE}_{d, \zeta}(\nu, \Sigma)$ distributed according to
the multivariate power exponential distribution with power $\zeta>0$,
location $\mu$, scale $b$, shape $\Sigma\in \R^d\times \R^d$ and pdf
\begin{equation}
\varphi_{\zeta}(x)= a_{d, \zeta} \left|\Sigma \right|^{-1/2} \mathrm{exp}\left(
    - b \left( (x-\nu)^T \Sigma^{-1}(x-\nu)  \right)^{\zeta}\right)  
\label{eq:18}
\end{equation}
on $\R^d$ ($a_{d, \zeta}$ is the normalizing constant), $\zeta \in (0,
\infty)$, and let
$\mu(dx) = \varphi_{\zeta}(x) dx$ be the corresponding probability
measure.  
Clearly
$Z \sim \mathrm{E}_{d}(\nu, \Sigma, \phi)$ with
$\phi(t) = e^{-b t^{\zeta}}$ so that
$\phi'(t)/\phi(t) = - b\zeta t^{{\zeta-1}}$ and
$\phi''(t)/\phi'(t) = - b\zeta t^{{\zeta-1}} + \frac{\zeta-1}{t}$.
From \eqref{eq:elliptscore}, 
the score function of
$\varphi_{\zeta}$ is
\begin{equation}
  \label{eq:19}
\rho_{{\zeta}}(x) = -2 b \zeta ((x-\nu)^T
 \Sigma^{-1}(x-\nu))^{\zeta-1} \Sigma^{-1} (x-\nu).
\end{equation}
While, except when  $\zeta=1$,the kernel from
 \eqref{lemm:lemmasteik} does not lead to palatable expressions, applying 
\eqref{lem:steikellipt2} we obtain
(for $\zeta \neq 1$)
\begin{equation}\label{eq:26}
  \pmb \tau_{2, \zeta}(x) =  \frac{\beta + 2  \frac{\zeta-1}{b\zeta t^{\zeta}}}{(\beta-2)(d-1)} \left(  \left(
      1 -  \frac{d-1}{\beta b\zeta t^{\zeta} +2
        (\zeta-1)}\right)(x - \nu)^T\Sigma^{-1}(x - \nu) \Sigma - (x -
    \nu) 
(x - \nu)^T\right).
\end{equation}
These functions are Stein kernels : they are in the class $
\mathcal{F}(p)$, since $\pmb \tau_{2, \zeta} \,\varphi_\zeta \rightarrow 0$ when $x \rightarrow +\infty$.
\%eqref{eq:26}. We do not provide details here.

Note again that the Stein identity \eqref{eq:steinkernelidentity}
holds for every $g \in \mathcal C^1(\R^d)$ with at most polynomial
growth at infinity, since in this case we have
$\pmb \tau_{2,\zeta} \, g \, \varphi_\zeta \rightarrow 0$ at infinity.


\subsection{The multivariate Student $t$-distribution with $k > 1$}
\label{sec:mult-stud-t}

Consider a $d$ random vector $X \sim t_{k}(\nu, \Sigma)$ distributed
according to the multivariate Student-$t$ distribution with $k >1 $
degrees of freedom, location $\nu\in \R^d$, shape
$\Sigma\in \R^d\times \R^d$ and pdf 
\begin{equation}
  \label{eq:30}
  t_k(x) = c_{k, d } |\Sigma|^{-1/2} \left[
    1+\frac{(x-\nu)^T\Sigma^{-1}(x-\nu)}{k}\right]^{-(k+d)/2} 
\end{equation}
with normalizing constant
$c_{k, d, \Sigma} =
\Gamma((k+d)/2)/(\Gamma(k/2)k^{d/2}\pi^{d/2})$. Let
$\mu(dx)= t_k(x) dx$ be the corresponding probability measure. The
assumption that $k>2$ ensures that this distribution has finite mean
and finite variance.

This distribution is an elliptical distribution with
$\phi(t) = (1+2t/k)^{-(k+d)/2}$ and hence
$\phi'(t)/\phi(t) = -(k+d)/(k+2t)$ leading to the score function
  \begin{equation} 
  \label{eq:43}
\rho_{t_k}(x) = -(k+d)\frac{\Sigma^{-1}(x-\nu)}{k+(x-\nu)^T\Sigma^{-1}(x-\nu)}.
\end{equation}
Moreover, 
  \begin{equation*}
      \frac{1}{\phi(t)} \int_t^{+\infty}  \phi(u) du = \frac{k+2t}{d+k-2}
  \end{equation*}
  (from $k > 1$ it follows that $d+k>2$) and hence Lemma
  \ref{lemm:lemmasteik} gives that
\begin{equation}
  \label{eq:steikstud1}
  \pmb \tau_1(x) =  \frac{(x-\nu)^T\Sigma^{-1}(x-\nu)+k}{d+k-2} \Sigma 
\end{equation}
is a Stein kernel for the multivariate Student distribution for
$k > 1$, as then $ \pmb \tau_1 \in \mathcal{F} (t_k)$.

Similarly, using that  $\phi''(t)/\phi'(t) = -(d+k+2)/(k+2t)$, Lemma
\ref{lem:steikellipt2} gives a family of Stein 
kernels which are indexed by $\beta \in \RR$: 
 \begin{align*}
  &   \pmb \tau_{2, \beta} (x) \\
& = 
\frac{\beta(d+k) - 4}{(d+k)( \beta-2)(d-1)} \left( 2 \left(
    \frac{(d-1)(k+2t)}{4-\beta(d+k)} + t 
\right) \Sigma - (x - \nu) (x - \nu)^T  \right) \\
 & =
\frac{\beta(d+k) -4 }{(d+k)( \beta-2)(d-1)} \left(  2 \left(
    \frac{(d-1)k +t \left( 2(d + 1)-\beta(d+k) \right))}{4-\beta(d+k)}  
\right) \Sigma - (x - \nu) (x - \nu)^T  \right).
   \end{align*}
   It is easy to verify that
   $ \pmb \tau_{2, \beta} \in \mathcal{F}(t_k).$ If we
   choose $\beta$ so that $(d+k)\beta = 2(d+1)$, i.e.
   $\beta = 2(d+1)/(d+k)$ then, after simplifications, we obtain for
   $k > 2$
 \begin{equation}
  \label{eq:steikstu2}
  \pmb \tau_2(x) = \frac{1}{k-1} \left(  (x - \nu) (x - \nu)^T  + k \Sigma \right).
\end{equation}
Since $\pmb\tau_1$ and $\pmb \tau_2$ are of order $t$ for large $|x|$,
and since
$\phi(t) \underset{t\rightarrow +\infty}{\sim} t^{-(k+d)/2}$, the
Stein identity \eqref{eq:steinkernelidentity} holds for every
$g \in \mathcal C^1(\R^d)$ such that
$t^{-(k+d-2)/2} g(x) \rightarrow 0$ at infinity. In particular constant functions verify \eqref{eq:steinkernelidentity} and the Stein kernels are in $\mathcal F(p)$.
Note that both
$\pmb \tau_1$ and $\pmb \tau_2$  simplify to $\tau(x)= (x^2+k\sigma^2)/(k-1)$
when $d=1$; this last quantity is well-known to be the univariate
kernel for the Student-$t$ distribution with $k$ degrees of freedom
and centrality parameter $\nu$, see e.g.\ \cite[page 30]{LRS16}.

There are several types of
operators and identities that can be obtained; {below are some examples.}
\begin{enumerate}
\item {\bf Vector valued operators.}
  Applying the product rule \eqref{eq:prodrulenabla} directly with
  $f(x) = k + (x-\nu)^T\Sigma^{-1}(x-\nu)$ we obtain for
  $g :  \R^d \rightarrow \R \in \mathrm{dom}(f, t_k)$ the operator
  $\mathcal{A}_{t_k}g(x) = (k + (x-\nu)^T\Sigma^{-1}(x-\nu)) \nabla
  g(x) + (2-k-d) \Sigma^{-1}(x-\nu)g(x)$. Taking expectations for 
 $X \sim t_k(\nu, \Sigma)$ we obtain the vector-Stein identity
  \begin{equation}
    \label{eq:47}
   \E\left[  (k + (X-\nu)^T\Sigma^{-1}(X-\nu)) \nabla
  g(X)\right] = (k+d-2) \E \left[\Sigma^{-1}(X-\nu)g(X) \right]
  \end{equation}
By definition of the Stein kernel we also get new Stein operators and
identities. Using $\pmb \tau_1$ recovers \eqref{eq:47}, whereas  using $\pmb \tau_2$ we
obtain 
\begin{equation}
  \label{eq:48}
  \E \left[ ((X-\nu) (X-\nu)^T +k \Sigma )\nabla g(X)  \right] =
  (k-1)\E \left[ (X-\nu) 
    g(X)\right]
\end{equation}
(still with $X \sim t_k(\nu, \Sigma)$). 
\item {\bf Scalar valued operators.} 
  Suppose for simplicity that $\Sigma= I_d$ and $\nu=0$. Taking
  $\mathbf{B}$ successively equal to $\pmb \tau_1$ then $\pmb \tau_2$ in \eqref{eq:axfnuhess}
 leads to
\begin{align*}
&  \mathcal{B}_1g(x) =   - \langle \nabla g(x), x \rangle  +
  \frac{1}{d+k-2}  \left\langle x^Tx+2k, \nabla^2g(x)  
\right\rangle_{\rm{HS}} \\ 
 & \mathcal{B}_2g(x) = 
- \langle \nabla g(x), x \rangle  + \frac{1}{k-1}  \left\langle x x^T +k
  I_d, \nabla^2 g(x) 
\right\rangle_{\rm{HS}}
\end{align*}
acting on functions $g$ such that $\nabla g \in \mathcal{F}_1(t_k)$.
\item {\bf A covariance identity.} 
Starting from \eqref{eq:48} with
  $X = (X_1, X_2)^T \in \R^{d_1}\times \R^{d_2}$ multivariate student
  with location $(\nu_1, \nu_2)$ and shape $
  \begin{pmatrix}
    \Sigma_{11} & \Sigma_{12} \\
\Sigma_{21} & \Sigma_{22}
  \end{pmatrix}
  $, taking $g(x) = g(x_2)$ and considering only the first $d_1$
  components of the resulting identity we obtain 
\begin{align}
  \label{eq:51}
 \E \left[
    ((X_1-\nu_1)(X_2-\nu_2)^T + k \Sigma_{12})\nabla g(X_2)\right]
= (k-1)\E \left[
    (X_1-\nu_1)g(X_2) .
\right]
\end{align}
Many more such covariance identities can be obtained by this approach,
thus complementing those obtained in \cite{adcock2007}. 
\end{enumerate}

\section{Generalities on  Stein kernels}
\label{sec:stein-discrepancies-1}
\label{sec:stein-kernel}

Let $\mathcal{T}_{\mathrm{div}, p}$ be the canonical
Stein operator \eqref{eq:42} acting on $\mathcal{F}(p)$, the
corresponding Stein class.  In this section we explore properties of
the Stein kernels from Definition \ref{def:multiteikdef}.

\begin{prop}[Properties] \label{prop:properties}
Let  $\tau_{p, i}$  be a  Stein kernel for $p$ in the direction
$e_i$ and $\pmb \tau_p$ the matrix with $i^{th}$ raw being $\tau_{p, i}$.  Then 
  \begin{enumerate}
  \item \label{item:1}   For all $j=1, \ldots, d$ we have 
    \begin{equation*}
    \frac{\partial}{\partial x_j} \left( \int_{\R^d}
\tau_{p, i}(x) p(x) dx \right) = \int_{\R^d}
    \frac{\partial}{\partial x_j}(\tau_{p, i}(x) p(x)) dx  =0  .     
    \end{equation*}
  \item \label{item:2} If $p$ admits a second moment, and if
    $x-\nu \in \mathrm{dom}(p,\pmb \tau_{{p}})$, then
    \begin{equation*}
      \E \left[ \pmb \tau_{p}(X)\right] = \mathrm{Var}(X).
    \end{equation*}
  \end{enumerate}
  \end{prop}

  \begin{proof}
    The first statement follows by the requirement that the kernel
    belongs to $\mathcal{F}(p)$, which in particular imposes that all
    components of $\tau_{p, i}$ belong to $\mathcal{F}_1(p)$. To see
    the second claim, taking expectations in \eqref{eq:29} yields that
    the Stein kernel
    necessarily satisfies 
    \begin{equation}\label{eq:34}
      \mathbb{E} \left[ \pmb \tau_p(X) \nabla g(X) \right] = \E \left[ (X-\nu)g(X) \right]
    \end{equation}
    for all $g : \R^d \to \R$ belonging to $\mathrm{dom}(p, \pmb \tau_{p})$. By
    assumption, $g_i(x) = x_i - \nu_i$ belongs to
    $\mathrm{dom}(p, \pmb \tau_p)$ for all $i=1, \ldots, d$. Plugging these
    functions in \eqref{eq:34} leads to
\begin{equation*}
    \mathbb{E} \left[ (\pmb \tau_{p,i}(X))_j  \right] =\E \left[ (X_i-\nu_i)(X_j-\nu_j) \right]
\end{equation*}
for all $i, j = 1, \ldots, d$. The claim follows. 
  \end{proof}

\begin{prop} Given $k\le d$ and
  $\left\{ i_1, \ldots, i_k \right\}\subset \left\{ 1, \ldots, n
  \right\}$ denote by
  $\mathcal{V} = \left\langle e_{i_1}, \ldots, e_{i_k} \right\rangle$
  the space generated by $e_{i_1}, \ldots, e_{i_k}$. Also, write any
  $x\in \R^d$ as $x = (x^{\mathcal{V}}, x^{\mathcal{V}^{\perp}})$ and
  let $p_{\mathcal{V}} = \int_{\mathcal{V}^{\perp}} p$ be the marginal
  of $p$ on $\mathcal{V}$. Suppose that $p$ admits a $p$-integrable
  Stein kernel
  $\tau_{i_j}= \big(\tau_{i_j,1}, \ldots, \tau_{i_j, d}\big)$ in each
  direction $e_{i_1}, \ldots, e_{i_k}$. Then the vector
  $\tau^{\mathcal{V}}_{i_j}= \big(\tau^{\mathcal{V}}_{i_j, 1}, \ldots,
  \tau^{\mathcal{V}}_{i_j ,k}\big)$ with components
  \begin{equation}
    \label{eq:72}
    \tau^{\mathcal{V}}_{i_j, \ell}(x^{\mathcal{V}}) =\EE\left[ \tau_{i_j,
        i_{\ell}} (X)  \, | \,
      X^{\mathcal{V}} = x^{\mathcal{V}} \right], \quad
    \ell = 1, \ldots, k
  \end{equation}
is a Stein kernel for $p_{\mathcal{V}}$ in the direction $e_{i_j}$. 
\end{prop} 
\begin{proof}
  Without loss of generality we suppose that $p$ is centered.  Fix
  $x^{\mathcal{V}}\in \mathcal{V}$. Then 
\begin{align*}
 &  \sum_{\ell=1}^k \frac{\partial}{\partial_{i_\ell}}  \left( \tau^{\mathcal{V}}_{i_j,\ell}
  (x^{\mathcal{V}}) p_{\mathcal{V}}(x^{\mathcal{V}}) \right) \\
& =
 \sum_{\ell=1}^k  \frac{\partial}{\partial_{i_\ell}} \left(  \int_{\mathcal{V}^{\perp}}
  \tau_{i_j, i_{\ell}}(x^{\mathcal{V}}, x^{\mathcal{V}^{\perp}})
p(x^{\mathcal{V}},
  x^{\mathcal{V}^{\perp}}) d  x^{\mathcal{V}^{\perp}} \right)\\
& =
  \sum_{\ell=1}^k \int_{\mathcal{V}^{\perp}} \frac{\partial}{\partial_{i_\ell}}  \left( 
  \tau_{i_j, i_{\ell}}(x^{\mathcal{V}}, x^{\mathcal{V}^{\perp}})
p(x^{\mathcal{V}},
  x^{\mathcal{V}^{\perp}})  \right)d  x^{\mathcal{V}^{\perp}} \\
  & = -
x_{i_j}\int_{\mathcal{V}^{\perp}}p(x^{\mathcal{V}},
  x^{\mathcal{V}^{\perp}})   d  x^{\mathcal{V}^{\perp}}  + 
\sum_{m \in \left\{ 1, \ldots, n \right\} \setminus \left\{ i_1, \ldots,
                                                               i_k
                                                               \right\}}
  \int_{\mathcal{V}^{\perp}}  \frac{\partial}{\partial_m}  \left(  
  \tau_{i_j,m}(x^{\mathcal{V}}, x^{\mathcal{V}^{\perp}})
p(x^{\mathcal{V}},
  x^{\mathcal{V}^{\perp}})\right) d  x^{\mathcal{V}^{\perp}}
\end{align*}
where the second-last line is allowed thanks to integrability of the
Stein kernel and the last follows from \eqref{eq:65}, giving the
identity
\begin{align*}
 \sum_{m=1}^d \frac{\partial}{\partial_m}  \left( 
  \tau_{i_j,m}(x^{\mathcal{V}}, x^{\mathcal{V}^{\perp}})
p(x^{\mathcal{V}},
  x^{\mathcal{V}^{\perp}})\right)  = - x_{i_j} p(x^{\mathcal{V}},
  x^{\mathcal{V}^{\perp}})
\end{align*}
which is valid for any $ x^{\mathcal{V}^{\perp}}$ such that
$(x^{\mathcal{V}}, x^{\mathcal{V}^{\perp}})$ lies in the support of
$\Omega_p$.  Now using the requirement $\tau_{i_j} \in \mathcal{F}(p)$
ensures that
$ \int_{\mathcal{V}^{\perp}} \frac{\partial}{\partial_m} \left(
  \tau_{i_j,m}(x^{\mathcal{V}}, x^{\mathcal{V}^{\perp}})
  p(x^{\mathcal{V}}, x^{\mathcal{V}^{\perp}})\right) d
x^{\mathcal{V}^{\perp}} = 0 $ for all
$m \notin \left\{ i_1, \ldots, i_k \right\}$ so that
$$\sum_{\ell=1}^k \frac{\partial}{\partial_{i_\ell}}  \left( \tau^{\mathcal{V}}_{i_j,\ell}
  (x^{\mathcal{V}}) p_{\mathcal{V}}(x^{\mathcal{V}}) \right) = -
x_{i_j} p_{\mathcal{V}}(x),$$ as required. It remains to check that
$\tau^{\mathcal{V}}_{i_j,\ell} (x^{\mathcal{V}})\in
\mathcal{F}(p_{\mathcal{V}})$, but this is a direct consequence of the
definitions.
\end{proof}

Next, we provide
formulas for computing Stein kernels explicitly based on univariate
Stein kernels.
\begin{prop}[Bivariate Stein kernels]\label{prop:transport-definition-1}
  Let $X = (X_1, X_2)^T\sim p$ with $p$ a continuous pdf on $\R^2$
  satisfying Assumption A. Let $p_1$ be the marginal of $p$ in
  direction $e_1$, $\rho_1$ the corresponding univariate score and
  $\tau_1$ the corresponding univariate kernel (which we suppose to
  exist). Set $\tau_{11}(x_1, x_2) = \tau_1 (x_1),$ and
\begin{align}
\label{eq:54}
  \tau_{12}(x_1, x_2) & =\tau_1(x_1) \frac{p_1(x_1)}{p(x_1, x_2)} \partial_{1}
   \left(  \int_{x_2}^{\infty}p(x_1, v)dv/p_1(x_1) \right) .
 \end{align}
 Then the vector
 $(x_1, x_2) \mapsto \left( \tau_1(x_1, x_2), \tau_{12}(x_1, x_2)
 \right)_{1 \le i, j\le 2} $ is a Stein kernel for $p$ in the
 direction $e_1$.
\end{prop}
Note that 
\begin{align}
\lefteqn{\tau_1(x_1) \frac{p_1(x_1)}{p(x_1, x_2 )} \partial_{1}
   \left(  \int_{x_2}^{\infty}p(x_1, v)dv/p_1(x_1) \right) } \nonumber \\  &= \tau_1(x_1) \frac{p_1(x_1)}{p(x_1, x_2)} \left( -\frac{p_1'(x_1)}{(p_1(x_1))^2}
     \int_{x_2}^{{\infty}} p(x_1, v) dv  + \frac{1}{p_1(x_1)}
     \int_{x_2}^{{\infty}}\partial_1 p(x_1,v) dv \right)\nonumber \\
  \label{eq:67}
& = \frac{1}{p(x_1, x_2)}\tau_1(x_1)
     \int_{x_2}^{{\infty}} \big(  -\rho_1(x_1)p(x_1, v)  + 
\partial_1 p(x_1,v)   \big)dv.
 \end{align}
 This alternative form of \eqref{eq:54} is often more convenient than  \eqref{eq:54}.

\begin{proof} We need to prove that 
  \begin{align}
   & \sum_{j=1}^2 \partial_j(\tau_{1j}(x) p(x)) = -(x_1
     -\nu_1)p(x) \label{eq:59}  
\end{align}
for all $x  = (x_1, x_2)^T\in \R^2$.
Applying \eqref{eq:67} we have 
  \begin{align}
&    \tau_{12}(x)  p(x_1, x_2)  = p_1(x_{1})\tau_1(x_1)\partial_1
    \left( \int_{x_2}^{{\infty}}p(x_1, v) dv/ p_1(x_1) \right)
  \nonumber \\
 & \qquad = -\tau_1(x_1)\frac{p_1'(x_1)}{p_1(x_1)}
     \int_{x_2}^{{\infty}} p(x_1, v) dv  + \tau_1(x_1)
     \int_{x_2}^{{\infty}}\partial_1 p(x_1,v) dv \label{eq:52}
  \end{align}
so that 
\begin{align}\label{eq:57}
  \partial_2 \left(\tau_{12}(x)p(x_1, x_2)) \right) & = 
                                               \tau_1(x_1)\frac{p_1'(x_1)}{p_1(x_1)} 
                                               p(x_1, x_2) -
                                               \tau_1(x_1) \partial_1p(x_1, x_2).  
\end{align}
 Also,  as $\partial_1(p_1(x_1)\tau_1(x_1))   = (\nu_1 - x_1) p_1(x_1)$,
 \begin{align}
&  \partial_1(\tau_{1}(x) p(x_1, x_2))
  =  \partial_1\left(p_1(x_1)\tau_1(x_1) \frac{p(x_1, x_2)}{p_1(x_1)}\right)
  \nonumber\\ 
   &\qquad  = \partial_1(p_1(x_1)\tau_1(x_1))  \frac{p(x_1, x_2)}{p_1(x_1)} +
     p_1(x_1)\tau_1(x_1) \partial_1 \left(  \frac{p(x_1, x_2)}{p_1(x_1)} \right)\nonumber\\
& \qquad=  -(x_1-\nu_1) p(x_1, x_2) +  p_1(x_1)\tau_1(x_1) \left(
  -\frac{p_1'(x_1)}{(p_1(x_1))^2} p(x_1, x_2)+
  \frac{\partial_1p(x_1, x_2)}{p_1(x_1)} \right)\nonumber\\
& \qquad = -(x_1-\nu_1) p(x_1, x_2)  - \tau_1(x_1) \frac{p_1'(x_1)}{p_1(x_1)}
  p(x_1, x_2) + \tau_1(x_1) \partial_1p(x_1, x_2). 
  \label{eq:58}
 \end{align}
 Adding up \eqref{eq:57} and \eqref{eq:58} we get \eqref{eq:59}.
\end{proof}
\begin{rem}
  The proof of Proposition \ref{prop:transport-definition-1} is of a
  purely computational nature. The inspiration for formula
  \eqref{eq:54} is \cite[equation (9)]{BaBaNa03}, where a similar
  quantity is introduced via a transport argument. To see the
  connection, 
note that   \eqref{eq:67} gives
  \begin{align}
    \label{eq:55}
    \tau_{12}(x_1, x_2) p(x_1, x_2)= \tau_1(x_1)
    \int_{-\infty}^{x_2}\left( \rho_1(x_1) p(x_1, v)
    - \partial_1p(x_1, v)  \right)dv.
  \end{align}
  We introduce $ p^{X \, | \, X_i=x_i}(x_1, x_2) = {p(x_1, x_2)}/{p_i(x_i)}$
  the conditional density of $X$ at $X_i=x_i$. Fix $i=1$ and, for each
  $t, t', x_2$ let $x_2 \mapsto T_{t, t'}(x_2)$ be the mapping
  transporting the conditional density at $x_1=t$ to that at $x_1=t'$,
  implicitly defined via 
\begin{align}
  \label{eq:56}
  p^{X \, | \, X_1=t}(x_2)  & = p^{X \, | \,
                                          X_1=t'}\big(T_{t,
                                          t'}(x_2)\big)
                              \partial_{x_2} T_{t, t'}({x_2}	).  
\end{align}
Taking derivatives in \eqref{eq:56} with respect to $t'$ and setting
$t'=t=x_1$ we deduce (using the fact that $T_{t, t}({x_2}	) = {x_2}$) that 
\begin{align*} 
  0 & = \frac{\partial_1 p(x_1, {x_2}	)}{p_1(x_1)} + \frac{\partial_{2}	p(x_1,
    {x_2}	)}{p_1(x_1)}\partial_{t'} T_{t, t'}({x_2}	)
      \big|_{t'=t=x_1} -
  \frac{p(x_1,{x_2}	)}{p_1(x_1)}\frac{p_1'(x_1)}{p_1(x_1)}\\
& \qquad  + 
  \frac{p(x_1,{x_2}	)}{p_1(x_1)} \partial_{{2}	}(\partial_{t'} T_{t, t'}({x_2}	))
  \big|_{t'=t=x_1},
\end{align*}
that is, 
\begin{align*}
  \partial_2 \left( p(x_1, x_2) \partial_{t'} T_{t, t'}({x_2}
  )\big|_{t'=t=x_1}\right) = p(x_1, x_2)\frac{p_1'(x_1)}{p(x_1)}
  - \partial_1 p(x_1, x_2) 
\end{align*}
and we recognize from  \eqref{eq:55} that, up to a
    function which depends only on $x_1$,  
\begin{align}
  \label{eq:68}
\frac{\tau_{12}(x_1, {x_2}	)}{\tau_1(x_1)} = \partial_{t'} T_{t, t'}({x_2}	)
  \big|_{t'=t=x_1}.
\end{align}
This is not the only Stein kernel in connection with transport maps,
see \cite{F2018} where yet another construction is introduced.
 \end{rem}

 Further, inspired by \cite{artstein2004solution,ABBNptrf}, we can
 directly postulate our next result which guarantees existence of
 Stein kernels under smoothness conditions. 

\begin{thm}\label{th:artstein} 
Let $p : \R^d \to (0,  \infty)$ be a 
  continuously twice differentiable density on $\R^d$ with 
  \begin{equation*}
    \int \frac{\| \nabla p \|^2}{p}, \quad \int \| \nabla^2(p) \|
    < \infty. 
  \end{equation*}
  Let $\tau_i^{(1)}, i = 1, \ldots, d$ be the marginal Stein
  kernels. Then, for any direction   $e_i, i=1, \ldots, d$ there
exists a Stein kernel for $p$ in direction $e_i$ 
\begin{equation*}
  \tau_{p,i}^{(d)}(x) =\tau_i^{(1)}(x_i)
  \begin{pmatrix}
    \tau_{i,1}^{(d)}(x \, | \, x_i) & \cdots & \tau_{i,i-1}^{(d)}(x \,
    | \, x_i) & 1 & \tau_{i,i+1}^{(d)}(x \,
    | \, x_i) & \cdots  & 
    \tau_{i,d}^{(d)}(x \, | \, x_i) 
  \end{pmatrix}
\end{equation*}
with coefficients $\tau_{i\bullet}, i=1, \ldots, d$ solving the
equations
  \begin{equation}\label{eq:60}
    \mathcal{T}_{\mathrm{div}, p}(\tau_{i\bullet}(x \, | \, x_i)) = \rho_i(x_i).
  \end{equation}
Here $\rho_i(x_i) = p_i'(x_i)/p_i(x_i)$  is the score function of the
$i$th marginal and $x=(x_1, \ldots, x_d)$. 
\end{thm}
\begin{proof}
  The result is almost immediate from \cite[Theorem
  4]{artstein2004solution} where it is proved (see middle of page 978)
  that, under the stated conditions, there exist continuously
  differentiable vector fields $R_h$ such that 
  \begin{align*}
    \frac{\mathrm{div}(R_h(x)p(x))}{p(x)} = \frac{h'(u)}{h(u)}
  \end{align*}
  for any marginal $u \mapsto h(u)$ of $p$, in any
  direction. Collecting these into a single vector and adapting the
  notations leads to \eqref{eq:60}.  To see the connection with Stein
  kernels, write
  \begin{equation}
    \label{eq:61}
    \tau_{ij}(x) = \tau_i(x_i) \tau_{ij}(x \, | \, x_i).  
  \end{equation}
Then 
\begin{align*}
  \sum_{j=1}^d \partial_j(\tau_{ij}(x) p(x)) & =
  \sum_{j=1}^d \partial_j(\tau_{ij}(x \, | \, x_i)p(x) \tau_i(x_i)) \\
  & =   \sum_{j=1}^d \partial_j(\tau_{ij}(x \, | \, x_i)p(x))   \tau_i(x_i) + 
    \sum_{j=1}^d \tau_{ij}(x \, | \, x_i)p(x) \partial_j(\tau_i(x_i)) \\
  & =    \rho_i(x_i)  p(x) \tau_i(x_i)+
p(x) \partial_i(\tau_i(x_i))
\end{align*}
where in the last line we use \eqref{eq:60} in the first sum and
$\partial_j(\tau_i(x_i)) = 0$ for all $j \neq i$ in the second sum.
Clearly by the definition of the univariate Stein kernel
\begin{align*}
   \partial_i(\tau_i(x_i)) = - \rho_i(x_i) \tau_i(x_{i})  +\E[X_i] - x_i
\end{align*}
and the claim follows. 
\end{proof}

\begin{rem}
Theorem \ref{th:artstein} gives a mechanism allowing to generalize the
bivariate construction from Proposition
\ref{prop:transport-definition-1}. Under the conditions of Theorem
\ref{th:artstein}, for $d=3$, we set 
  $p_{ij}(x_i, x_j) = \int_{-\infty}^{\infty} p(x) dx_k$ and $P_i(x) =
  \int_{-\infty}^xp_i(v)dv, i= 1, 2, 3$. Then we can
  choose $\tau_{i,i} = \tau_i$ and 
\begin{align*}
&  \tau_{1,2}^{(3)}(x \, | \, x_1)p(x_1, x_2, x_3) = 
   \int_{-\infty}^{x_2} \left(\rho_1(x_1) p(x_1, v, x_3) - \partial_1
  p(x_1, v, x_3) \right) dv   \\
& \qquad - P_2(x_2)
  \big\{ \rho_1(x_1) p_{13}(x_1, x_3) - \partial_1 p_{13}(x_1, x_3) \big\}\\
&  \tau_{1,3}^{(3)}(x \, | \, x_1)p(x_1, x_2, x_3)   = 
 p_2(x_2) \int_{-\infty}^{x_3} \left( \rho_1(x_1) p_{13}(x_1, w)
  -  \partial_{1}p_{13}(x_1, w) \right) dw
\end{align*}
  and similarly for $\tau_{i,j}^{(3)}(x \, | \, x_i)$ for all $i, j$.
 Direct computations suffice for this claim.
    Also, setting 
  \begin{equation}
   \label{eq:62}
 \tau_{12}^{(2)}(x_1, x_2) =   \tau_1^{(1)}(x_i)\E[\tau_{12}^{(3)}(x \, | \, x_1)
 \, | \, X_1=x_1 , X_2=x_2],
  \end{equation}
  the vector $(\tau_1^{(1)}(x_1), \tau_{12}^{(2)}(x_1, x_2))$ forms a
  bivariate Stein kernel for $(X_1, X_2)$.
%
 Moreover,
  \begin{align*}
    \int_{-\infty}^{+\infty} \tau_{1,2}^{(3)}(x \, | \, x_1)\frac{p(x_1, x_2,
    x_3) }{p(x_1, x_2)}dx_3 &  = \frac{1}{p(x_1, x_2)} 
\int_{-\infty}^{x_2} \left(\rho_1(x_1) p(x_1, v) - \partial_1
  p(x_1, v) \right) dv   \\
& \qquad - P_2(x_2)
  \big\{ \rho_1(x_1) p_{1}(x_1) - \partial_1 p_{1}(x_1) \big\}\\
    & = \frac{1}{p(x_1, x_2)} 
\int_{-\infty}^{x_2} \left(\rho_1(x_1) p(x_1, v) - \partial_1
  p(x_1, v) \right) dv
  \end{align*}
  which is equivalent to the expression \eqref{eq:67}.

The $k$-variate extension can also be constructed, as follows.
For all $k \ge 1$, and under the same conditions, for all $1 \le j \le d-1$ we can define  
\begin{align}
  \nonumber
  \tau_{1, j}^{(k)}(x \, | \, x_1)p(x)&  = \int_{-\infty}^{x_j} \left(
    \rho_1(x_1) p(x_1, x_j=v, \ldots, x_d) - \partial_1p(x_1, x_j=v,
    \ldots, x_d) \right) dv \\
  & \qquad - P_j(x_j) \left(  \rho_1(x_1) p(x_1,
   x_{j+1} \ldots, x_d) - \partial_1p(x_1, x_{j+1}, 
    \ldots, x_d)  \right)  \label{eq:63}
\end{align}
and for $j=d$
\begin{align}
  \label{eq:64}
  \tau_{1, d}^{(k)}(x \, | \, x_1)p(x)  =  P_{d-1}(x_{d_1})  \int_{-\infty}^{x_d}\left(  \rho_1(x_1) p(x_1, x_d=v) - \partial_1p(x_1,  x_d = v)  \right) dv
\end{align}
\end{rem} 
  
\begin{example}[Multivariate Gaussian]
  If $X \sim \mathcal{N}_2(\nu, \Sigma)$ is multivariate
  $d$-dimensional Gaussian then direct computations of the kernel as
  provided by Proposition \ref{prop:transport-definition-1} leads to
  $\pmb \tau_2$ given in \eqref{eq:steikstu2}. The expression is more
  complicated in dimension $d \ge 3$, and so far we have not been able
  to give a probabilistic interpretation of it.
\end{example}
\begin{example}
  If $X = (X_1, X_2)^T \sim t_k(\nu, \Sigma)$ follows the bivariate
  Student distribution then direct computations of the kernel as
  provided by Proposition \ref{prop:transport-definition-1} leads to
  $\pmb \tau_2$ given in \eqref{eq:steikstu2}. Again, we have not been able
  to give a probabilistic interpretation of the expression in
  dimension $d \ge 3$.
\end{example}

\section{Stein discrepancies} 
\label{sec:stein-discrepancies}

{Instead of using the Wasserstein metric which uses Lipschitz functions, more general classes functions $\mathcal{G}$ 
in $\sup_{g \in \mathcal{G}}\left| E[ \mathcal{A}_pg(Y)] \right|$ can be useful to assess distributional distances, leading to the notion of {\it Stein discrepancies}.}

\subsection{Integral probability metrics and Stein discrepancies}
\label{sec:integr-prob-metr}

{Differences between distributions can be measured using
  probability metrics. For applying Stein's method, so-called
  {\it{integral probability metrics}} are well suited.}

\begin{definition}[Integral Probability Metrics] \label{def:ipms} Let
  $\mathbb{F}(\R^d)$ be a collection of cumulative distribution
  functions on $\R^d$ and denote $L^1(\mathbb{F}(\R^d))$ the class of
  Borel measurable functions $h: \mathbb{R}^{d}\to \mathbb{R}$ (for
  some $q \ge 1$) such that $F(|h|)=\int|h|dF<\infty$ for all
  $F \in \mathbb{F}(\R^d)$. A metric on $\mathbb{F}(\R^d)$ is an
  \emph{integral probability metric} {(IPM)} if it can be written in the form
\begin{equation}\label{eq:IPM}
   d_{\mathcal{H}}(F, G) :=    \sup_{h \in \mathcal{H}} \left| F(h)-  G(h) \right|
\end{equation}
for some class of real-valued bounded measurable test functions
$\mathcal{H} \subset L^1(\mathbb{F}(\R^d))$ ($\left| \cdot \right|$ is
the Euclidean norm).  The expression on the right-hand side of
\eqref{eq:IPM} is called an {\emph{IPM-discrepancy}}. 
\end{definition}

\noindent Many important probability metrics can be represented as
integral probability metrics; classical references are
\cite{zolotarev1983probability,GS02}.  The Wasserstein distance
between $X$ and $Y$, which we have already used in this paper, takes
$\mathcal{H}= \mathcal{W} $ the collection of Lipschitz
functions $h : \R^d \to \R$ with Lipschitz constant smaller than
1. The Kolmogorov distance between two random vectors $X \sim F$ and
$Y \sim G$ is
$ \mathrm{Kol}(X, Y) = d_{\mathcal{H}_{\mathrm{Kol}}}(F, G) $ with
$\mathcal{H}_{\mathrm{Kol}} = \left\{ \mathbb{I}_{(-\infty, z]}, z \in
  \R^d \right\}$. The Wasserstein distance between $X$ and $Y$ takes
$\mathcal{H}=  \mathcal{W} $ the collection of Lipschitz
functions $h : \R^d \to \R$ with Lipschitz constant smaller than 1.
The total variation distance takes $\mathcal{H}_{\mathrm{TV}}$ the
collection of Borel measurable functions $h : \R^d\to[0,1]$. For other
examples and references see \cite[Appendix E]{NP11}.

{In Definition \ref{def:ipms} the Euclidean norm $ | \cdot | $ is
  used, but the definition generalises easily to other norms
  $\| \cdot \|$ as long as
\begin{equation}\label{genipm} 
   d_{\mathcal{H}}(F, G) :=    \sup_{h \in \mathcal{H}} \left\|
       F(h)- G(h) \right\|
\end{equation}
defines a distance between probability distributions. 
}
This intuition leads to the following general definition.

\begin{definition}[Stein discrepancy] \label{def:steidisck}
  Let $p$ be a density on $\R^d$  and  $\mathcal{A}_p$  a Stein operator
  acting on some class 
  $\mathrm{dom}(\mathcal{A}_p)$. Then for any random $Y\sim q$, any
  $\mathcal{G}\subset \mathrm{dom}(\mathcal{A}_p)$ and any 
  norm $\left\| \cdot \right\|$ ,
  the quantity
\begin{align}
  \label{eq:stediscrepancy}
  \mathcal{S}_{\left\| \cdot \right\|}(q, \mathcal{A}_p, \mathcal{G}) = \sup_{g \in
    \mathcal{G}} \left\| \E \left[ \mathcal{A}_pg(Y) \right] \right\|
\end{align}
is the ($\left\| \cdot \right\|-\mathcal{G}-\mathcal{A}_p$) \emph{Stein discrepancy} from $Y$
to $X$.
\end{definition}

\noindent Definition~\ref{def:steidisck} {is motivated by} the
reference \cite{gorham2015measuring} where, to the best of our
knowledge, such a unified notation for general Stein-based
discrepancies (with freedom of choice both in the operator and the
class of functions) is first introduced.

The choice of norm $\left\| \cdot \right\|$ is generally fixed by
context such as dimensionality, basic properties of the operator and the
random variables $X, Y$ under study. In the sequel we will
generally drop the indexation in the norm and simply write
$\mathcal{S}(Y, \mathcal{A}_p, \mathcal{G})$ instead. 
The
 next subsection links Stein discrepancies to information metrics.

\subsection{Information  metrics and kernelized Stein discrepancies }
\label{sec:kern-stein-discr}

In principle, the Stein discrepancy \eqref{def:steidisck} can be used
as a basis for goodness-of-fit tests, and could be estimated
numerically. In high-dimensional problems, the class $\mathcal{G}$ is
often too large to allow numerical integration.  In high-dimensional
goodness of fit tests, restricting the class of functions to a ball in
a reproducing kernel Hilbert space associated with a positive definite
kernel $k(x,x')$ has been shown in
\cite{chwialkowski2016kernel,liu2016kernelized} to be an efficient way
of estimating Stein discrepancies. In this context these discrepancies
are called {\emph{kernelized Stein discrepancies}}, with the kernel
$k(x,x')$ in mind. The framework of the present paper shows how to
generalise their approach, as follows.

Let $X\sim p$ and $Y \sim q$ be two random variables on $\R^d$ with
differentiable densities and respective Stein classes $\mathcal{F}(p)$
and $\mathcal{F}(q)$. Suppose, for simplicity, that both share the
same mean and the same support $\mathcal{S}$, satisfying Assumption A.
Fix $d\times d$ matrix valued functions
$\mathbf{A}_p \in \mathcal{F}(p)$ and
$\mathbf{A}_q \in \mathcal{F}(q)$, set
$\mathbf{a}_p = \mathcal{T}_{\mathrm{div}, p}(\mathbf{A}_p)$,
$\mathbf{a}_q = \mathcal{T}_{\mathrm{div}, q}(\mathbf{A}_q)$ and
introduce the divergence based vector valued standardizations
\begin{align} \label{Ap} 
&   \mathcal{A}_pg(x) =  \mathbf{A}_p(x) \nabla g(x)
  + 
   \mathbf{a}_p(x)  g(x), \quad g:\R^d\to \R \in \mathcal{F}(\mathcal{A}_p)
\\ \label{Aq} 
&   \mathcal{A}_qg(x) =  \mathbf{A}_q(x) \nabla g(x)
  + 
\mathbf{a}_q(x)  g(x), \quad g:\R^d\to \R \in \mathcal{F}(\mathcal{A}_q)
\end{align}
as in \eqref{eq:37}.

The Stein heuristic that if $p$ and $q$ are close, then
    $\E_{q \otimes q} [ \mathcal{A}_{p \otimes p} k(Y,Y')] $ should be
    small still holds, where $\mathcal{A}_{p \otimes p}$ is the
    concatenated operator $\mathcal{A}_p$ operating on functions
    $g: \R^{2d} \rightarrow \R$ by treating the first $d$ and the last
    $d$ components independently. It turns out that iterating the
    operator is a more elegant way to obtain a kernelized
    discrepancy. Write $\mathcal{A}_p^T$ for the transpose of the
    operator $\mathcal{A}_p$ in the matrix transpose sense. Then for
    any positive definite symmetric kernel $k$ with marginals in
    $\mathcal{F}(p)$, using \eqref{Ap},
\begin{align}\label{it} 
\mathcal{A}_p^T \mathcal{A}_p k(x,x')
&=\mathbf{A}_p(x)^T \nabla_{x}^T  \mathbf{A}_p(x')   \nabla_{x'}  k(x,x')
  + 
 \mathbf{A}_p(x)^T  \nabla_{x}^T  \mathbf{a}_p(x') k(x,x') \nonumber \\
 & +
 \mathbf{a}_p(x)^T   \mathbf{A}_p(x')   \nabla_{x'}  k(x,x')
  + 
 \mathbf{a}_p(x)^T   \mathbf{a}_p(x') k(x,x')
.
\end{align}
By conditioning on $X$ it is easy to see that 
$$ \E_{p \otimes p} [ \mathcal{A}_p^T \mathcal{A}_p k(X,X') ] =0.$$ 
The operator \eqref{it} has been used in \cite{chwialkowski2016kernel}
for the particular choice $\mathbf{A}_p(x) = I_d$ for which
$\mathbf{a}_{p} = \rho_p $, the score operator.  In this case,
evaluating \eqref{it} does not require knowledge of the normalising
constant for the density $p$ and is hence particularly attractive for
applications in Bayesian inference.  Equation \eqref{it} motivates our
general definition of kernelized Stein discrepancies. We use the
convention that $\mathcal{L}_{i}, i = 1, 2$ denotes the operator
$\mathcal{L}$ applied with respect to the $i$th variable of the
function $k(\cdot, \cdot)$.
  
  \begin{definition}[Kernelized Stein discrepancies]\label{sec:inform-metr-kern}
  Let $\mathcal{A}_p$ (resp., $\mathcal{A}_p$) be a Stein operator for
  $p$ (resp., for $q$) with class $\mathcal{F}(p)$ (resp.,
  $\mathcal{F}(q)$). Let $k$ be some kernel
  $k : \R^d \times \R^d \to \R$ such that $y \mapsto k(y, \cdot)$ and
  $y' \mapsto k(\cdot, y')$ belong to
  $\mathcal{F}(\mathcal{A}_p)\cap \mathcal{F}(\mathcal{A}_p)$.  The
$k$-kernelized Stein discrepancy from  $p$ to $q$ is
\begin{equation}
  \label{eq:20}
  \mathcal{S}(p, q, k) =  \E \left[ \mathcal{A}_{p,1}^T\mathcal{A}_{p,2}k(Y, Y')
    \right].
\end{equation} 
When $X\sim p$ and $Y \sim q$, in abuse of notation we also write 
$\mathcal{S}(X,Y , k)  = \mathcal{S}(p, q, k).$
\end{definition}

Our set-up allows for a combination of the operators from \eqref{Ap}
and \eqref{Aq} which is sometimes more suitable for the problem at
hand.  All classes of functions are designed to ensure that, for all
$g \in \mathcal{G}=\mathcal{F}(\mathcal{A}_p)\cap
\mathcal{F}(\mathcal{A}_q)$, we have
\begin{align}
&\E \left[ \mathcal{A}_pg(Y) \right]  =\E \left[ \mathcal{A}_pg(Y)
\right] -\E \left[ \mathcal{A}_qg(Y) \right] \nonumber \\
&  =\E \left[  
    \left( \mathbf{A}_p(Y) - \mathbf{A}_q(Y) \right) \nabla g(Y)
   \right]+\E \left[
  (\mathbf{a}_p(Y) -
 \mathbf{a}_q(Y)) g(Y)
  \right]\nonumber  \\
  & \label{eq:leyswangeneee}
 =:\E \left[  
    \mathbf{A}_{p/q}(Y)\nabla g(Y)
  \right]+\E \left[
  \mathbf{a}_{p/q} (Y) g(Y)
  \right].
\end{align} 
Thus in particular we can take
\begin{align}\label{choice1} 
\mathbf{A}_{p/q} \mbox{ and } \mathbf{a} _{p/q} \mbox{ in  \eqref{Ap}} \\
-\mathbf{A}_{p/q} \mbox{ and } - \mathbf{a}_{p/q} \mbox{ in  \eqref{Aq}}.
\label{choice2}
\end{align} 
Two particular choices of input matrices $\mathbf{A}_p$ and
$\mathbf{A}_{q}$ stand out:
\begin{itemize}
\item $\mathbf{A}_p(x) = \pmb \tau_p(x)$ and $\mathbf{A}_q(x) = \pmb \tau_q(x)$ for
  which $ \mathbf{a}_{p/q}=0$ and \eqref{eq:leyswangeneee} becomes 
  \begin{equation}
    \label{eq:leyswanappsteik}
\E \left[ \mathcal{A}_pg(Y) \right]  =   \E \left[   
    (\pmb \tau_p(Y) - \pmb \tau_q(Y)) \nabla g(Y)
    \right]  =:E \left[   
    \pmb \tau_{p/q}(Y) \nabla g(Y)
\right] 
  \end{equation}
\item  $\mathbf{A}_p(x) =\mathbf{A}_q(x)= I_d$  for
  which 
$\mathbf{A}_{p/q}=0$ and $\mathbf{a}_{p/q} = \rho_p - \rho_q$  and  \eqref{eq:leyswangeneee} becomes 
  \begin{equation}
    \label{eq:15}
   \E \left[ \mathcal{A}_pg(Y) \right] =\E \left[ 
        (\rho_p(Y) - \rho_q(Y))
  g(Y) \right] =: \E \left[   \rho_{p/q}(Y) g(Y) \right].   
  \end{equation}
\end{itemize}

The Stein operators \eqref{Ap} and \eqref{Aq} can also be applied jointly to functions $k(x,x')$ with marginals in $\mathcal{G}$.  
\begin{align*}
 \mathcal{A}_{p,1}^T\mathcal{A}_{p,2}k(x,x') 
& =  \mathbf A_p^T (x)  \mathbf A_q(x') \nabla_x^T \nabla_{x'} k(x,x') + A_p^T (x) \mathbf{a}_q(x')  \nabla_{x}^T  k(x,x') \\
& +  \mathbf a_p(x)^T \nabla_{x'}  \mathbf A_q(x') k(x,x') +  \mathbf a_p(x)^T  \mathbf a_q(x') k(x,x').
\end{align*} 
In \cite{chwialkowski2016kernel}, such kernelized expressions are
studied for the score function choice \eqref{eq:15}.  Inspired by
\cite{chwialkowski2016kernel,liu2016kernelized} we give the following
result which follows immediately from \eqref{eq:leyswangeneee} with
the choices \eqref{choice1} and \eqref{choice2}. This result shows
that kernelized methods have a larger range of applicability than
usually assumed. At the same time, it illustrates the power of our
general set-up.

\begin{theorem}\label{theo:inform-metr-kern-1}
  Let $Y, Y'$ be \emph{independently} drawn from $q$ on the same space
  and consider functions $k : \R^d \times \R^d \to \R$ such that both
  $y \mapsto k(y, \cdot)$ and $y' \mapsto k(\cdot, y')$ belong to
  $\mathcal{F}(\mathcal{A}_p)\cap\mathcal{F}(\mathcal{A}_q)$.  Then
  \begin{align}
    &  \E    \left[ \mathcal{A}_{p,1}^T\mathcal{A}_{p,2}k(Y, Y')
      \right]\nonumber \\
    &  = \nonumber 
     \E \left[ \nabla_{y}^T \left( \mathbf{A}_{p/q}(Y)\mathbf{A}_{p/q}(Y') \right)
      \nabla_{y'} k(y,y') |_{y=Y, y'=Y'} 
      \right] \\
    & \qquad \nonumber 
      +\E \left[ \nabla_{y}^T \left(  \mathbf{A}_{p/q}(Y)
      \mathbf{a}_{p/q}(Y')k(y, Y')\right)|_{y = Y}  \right]
      +\E \left[ \mathbf{a}_{p/q}(Y)^T \nabla_{y'} \left(  \mathbf{A}_{p/q}(Y)
      k(y, Y') \right)|_{y = Y} \right] \\
    & \qquad \label{eq:genkernelinf}
+ \E \left[  \mathbf{a}_{p/q}(Y)^T k(Y, Y')\mathbf{a}_{p/q}(Y')\right]
  \end{align}
  with $\mathcal{L}_{i}, i = 1, 2$ the operator $\mathcal{L}$ applied
  with respect to the $i$th variable of the function
  $k(\cdot, \cdot)$.
\end{theorem}
Theorem \ref{theo:inform-metr-kern-1}, with the particular choice that
$\mathbf{A}_p(x) =\mathbf{A}_q(x)= I_d$, so that $\mathbf{A}_{p/q}=0$
and $\mathbf{a}_{p/q} = \rho_p - \rho_q$, was proposed in \cite{
  liu2016kernelized}. If both $p$ and $q$ are known up to a
normalising constant, then \eqref{eq:genkernelinf} does not depend on
the normalising constant and hence, again, is particularly attractive
in high-dimensional Bayesian inference.


From \eqref{eq:leyswanappsteik} and \eqref{eq:15} we introduce the
following two {special cases of particular interest.}
\begin{enumerate} 
\item If $\mathbf{A}_p(x) = \pmb \tau_p(x)$ and $\mathbf{A}_q(x) = \pmb \tau_q(x)$ for
  which $ \mathbf{a}_{p/q}=0$ then 
\begin{equation}
    \label{eq:kernelbasedkernelized}  
   \E \left[ \nabla_{y}^T \left( \pmb \tau_{p/q}(Y)\pmb \tau_{p/q}(Y') \right)
      \nabla_{y} k(t,t') |_{y=Y, y'=Y'} 
    \right]= \E
    \left[ \mathcal{A}_{p,1}^T\mathcal{A}_{p,2}k(Y, Y')  \right].
\end{equation}
\item If $\mathbf{A}_p(x) =\mathbf{A}_q(x)= I_d$ for which
  $\mathbf{A}_{p/q}=0$ and $\mathbf{a}_{p/q} = \rho_p - \rho_q$
then 
\begin{equation}
    \label{eq:scorebasedkernelized}
   \E \left[ \rho_{p/q}(Y)^T k(Y, Y') \rho_{p/q}(Y') \right] =\E
    \left[ \mathcal{A}_{p,1}^T\mathcal{A}_{p,2}k(Y, Y')  \right]
  \end{equation}

\end{enumerate}

 

\begin{example}[{Fisher information distance}]\label{ex:fisherinfo}
Pick  $k(y, y') = \delta_{y=y'}$ the
Dirac delta on the diagonal. Then \eqref{eq:20} with the choice
\eqref{eq:scorebasedkernelized} becomes  
\begin{equation}
  \label{eq:21}
  \mathcal{S}(Y, X, \delta) =  
  \E \left[ \rho_{X/Y}(Y)^T
    \rho_{X/Y}(Y) \right]  =:   J(Y/X) 
\end{equation}
with $J(Y/X)$ the  classical \emph{Fisher Information Distance} between $X$
and $Y$, see \cite{Jo04}.  
\end{example}

\begin{example}[{Independent kernels}]
  Let $(e_i)_{i=1, \ldots, n}$ be a sequence of functions in
  $\mathcal{F}(X)$ and
  $k(x, y) = \sum_{i=1}^d \alpha_i e_i \otimes e_i$ (which belongs to
  $\mathcal{G}$ for any $(\alpha_i)_{i=1, \ldots, n}$). Then
\begin{equation}
  \label{eq:stekindepkern}
   \mathcal{S}(Y, X, (\alpha, e)_n ) =  \sum_{i=1}^n \alpha_i(\E \left[
       \mathcal{A}_pe_i(Y) \right])^T(\E \left[ \mathcal{A}_pe_i(Y)
     \right]).
\end{equation}
 
\end{example}

\begin{example}(Kernelized Stein discrepancies for comparing Gaussian
  random vectors)
\label{sec:kern-stein-discr-1}
Let $X$ and $Y$ be independent centered multivariate normal random
variables in $\R^d$ with variances $\Sigma_1$ and $\Sigma_2$,
respectively.  Take $\mathcal{A}_pg(x) = \Sigma \nabla g(x) - x
g(x)$. Then, for any sufficiently regular function $e : \R^d \to \R$
we have
\begin{align*}\E \left[ \mathcal{A}_pe(Y)\right] =E \left[( \Sigma_1 -\Sigma_2) \nabla e(Y)\right] 
\end{align*}
so that the kernelized discrepancy \eqref{eq:stekindepkern} becomes 
\begin{align*}
  \mathcal{S}(Y, X, (\alpha, e)_n) =  \sum_{i=1}^n \alpha_i\E
  \left[ \nabla e_i(Y) \right]^T (\Sigma_1 - \Sigma_2)^2 \E
  \left[ \nabla e_i(Y) \right].
\end{align*}
Taking $n=d$, $\alpha_i=1$ and $e_i(y) = y_i$ and supposing that all
marginals have unit variance leads to the natural measure of
discrepancy
\begin{align*}
  \mathcal{S}_n(Y, X) = 2 \sum_{i<j} (\sigma_{ij}^X-\sigma_{ij}^Y)^2
\end{align*}
with $\sigma_{ij}^X$ (resp., $\sigma_{ij}^Y$) the covariance between
the marginals $i$ and $j$ of $X$ (resp., of $Y$). 
\end{example} 
  
In terms of potential applications, one of the most interesting
aspects of identity \eqref{eq:genkernelinf} is the fact that the right-hand side
justifies the use of the left-hand side as a discrepancy metric.
Applications of \eqref{eq:scorebasedkernelized} have begun to be
explored \cite{chwialkowski2016kernel,liu2016kernelized}, and more
general versions have been touched upon in
\cite{gorham2017measuring}. The freedom of choice in the input
matrices $\mathbf{A}_p, \mathbf{A}_q$ encourages us to be hopeful that
these quantities will have numerous applications. 

Inspired by \cite{chwialkowski2016kernel,liu2016kernelized}, we
conclude the section with an illustration of how to apply
    Stein discrepancies to obtain a goodness-of-fit test for
a Student-$t$ distribution.


\begin{example}[Dimension 1]
  Let $p$ be the centered Student-$t$ distribution with $\ell$ degrees of
  freedom, with score function  
  given by 
  \begin{equation}
    \label{eq:22}
    \rho_{t_\ell}(y) = -\frac{y(\ell+1)}{\ell+y^2} {{.}} 
  \end{equation}
The operator obtained from \eqref{eq:22} is
\begin{align}
&   \mathcal{A}_{t_\ell}^1f(y) = f'(y) - \frac{(\ell+1)y }{\ell+y^2}f(y)\label{eq:scorkernelt1}
\end{align}
The preceding developments  lead to postulating the sample-based
discrepancy
\begin{align*}
  &\mathcal{S}^{\rho}(t_\ell, q, k) =  \frac{1}{n_1n_2} \sum_{i= 1}^{n_1} \sum_{j=1}^{n_2}
    u_q^{\nu}(y_i, y'_j), 
\end{align*}
with $y_1, \ldots, y_{n_1}$ and $y'_1, \ldots, y'_{n_2}$ two i.i.d.\
samples independently drawn from $q$ and 
\begin{equation}
  \label{eq:uqgener1}
  u_q^{\rho}(y, y') = \mathcal{A}_{t_\ell, y}^\rho \mathcal{A}_{t_\ell,
    y'}^\rho  k(y, y'), 
\end{equation}
and $k(y, y')$ a well chosen kernel. Particularizing
\eqref{eq:uqgener1} we get 
\begin{align*}
 &  u_q^1(y, y') = \ \partial_y\partial_{y'}k(y, y')    
 - (\ell + 1) \frac{y'}{((y')^2+
  \ell)}  \partial_yk(y, y')   \\
& \qquad -  (\ell + 1) \frac{y}{(y^2+
  \ell)}  \partial_{y'}k(y, y')  
   +  (\ell + 1)^2 \frac{yy'}{(y^2+\ell)((y')^2+
  \ell)} k(y, y').
\end{align*}
Let $Y, Y' \sim q$ be two independent
copies. Since under natural conditions,
$\E_q \left[ u_q^{{{\rho}}}(Y, Y') \right] = 0$ if and only if $q = p$, a
natural goodness-of-fit test in this context  is to reject the
null assumption $ \mathcal{H}_0: q = p$ whenever
$\mathcal{S}^{\nu}(t_\ell, q, k)$ is too large.

For the sake of proof of concept rather than anything else,
here are the result of simulations comparing $X \sim p$
a Student with $\ell = 5 $ degrees of freedom with $Y \sim q$ a
Student with $\ell$ degrees of freedom, via the kernelized
discrepancies based on the RBF kernel $k(x, y) = e^{-(x-y)^2/2}$. The
quantiles for $\mathcal{S}^{\rho}(t_\ell, q, k)$ under the null
hypothesis were estimated by simulation, with $J=10^{5}$ experiments;
we obtained
\begin{center}
  \begin{tabular}{cc}
   2.5\% &       97.5\%  \\
\hline
-0.03837828 &   0.03970307 
  \end{tabular}
\end{center}
The results for $10^4$ simulations with $n_1 = n_2 = 100$ run for each
value of degrees of freedom
$\ell \in \left\{ 1, 4, 5, 6, 8, 10, 12, 100, 1000 \right\}$ (with
$\ell = 5$ corresponding to the null hypothesis) are reported below
(first line) as well as the corresponding results for the classical
Kolmogorov Smirnov test (R implementation \texttt{ks.test}), each time
on the same data:
\begin{center}
{\small
\begin{tabular}{c|cc|c|ccccccccccc}
  $\ell$ &  1 &   4 &  5 & 6   & 8   & 10   & 12 & 100  & 1000 \\
  \hline 
  kernel  & 0.9049&  0.0576 &  0.0507 & 0.0445  &  0.0433 & 0.0550
                                            & 0.0587 & 0.1330 & 0.1510 
  \\
  ks   & 0.9384 & 0.0497  & 0.0459  & 0.0479  & 0.0443  & 0.0461
                                            & 0.0480 & 0.0593& 
             0.0618 \\
  \hline
\end{tabular}
}
\end{center}
It appears that the test based on $u^{{\rho}}_q(\cdot, \cdot)$ is not as
powerful as the Kolmogorov Smirnov test, at least in our
 implementation. The numerical values
are not reported. A more detailed study of such Stein-based
discrepancy tests is under way (\cite{ERS18}).
\end{example}
\begin{example}[Dimension 2]
  Fix $d=2$ and let $p$ be the centered Student-$t$ distribution with
  $\ell$ degrees of freedom and $\Sigma = Id$ the identity matrix. 
We only consider the
operator
\begin{align*}
  &\mathcal{A}_1f(y) = \pmb \tau(y)\nabla f(y) - y f(y)
\end{align*}
with $\pmb \tau$ the Stein kernel matrix given in \eqref{eq:steikstu2},
with entries $(\tau_{ij})_{1 \le i, j \le 2}$; the
resulting kernelized discrepancy is
\begin{align*}
  &\mathcal{S}^{\nu}(t_\ell, q, \ell) =  \frac{1}{n_1n_2} \sum_{i= 1}^{n_1} \sum_{j=1}^{n_2}
    u_q(y_i, y'_j).
\end{align*}
with  discrepancy generating  function
\begin{align*}
  u_q(y, y')  =&    \mathcal{A}_{p,1}^TA_{p,2}k(y, y')  \\
             =  &  [ \tau_{11}(y) \tau_{11}(y') + \tau_{12}(y) \tau_{12}(y') ]\partial_{y_1} \partial_{y_1'}k(y, y') \\
               & \quad  +  [\tau_{11}(y) \tau_{12}(y') + \tau_{12}(y)
                 \tau_{22}(y') ] \partial_{y_1} \partial_{y_2'}k(y, y')\\
               & \quad  +    [\tau_{12}(y) \tau_{11}(y') +
                 \tau_{22}(y) \tau_{12}(y')]
\partial_{y_2} \partial_{y_1'}k(y, y') \\  
               & \quad  +  [\tau_{12}(y) \tau_{12}(y') + \tau_{22}(y)\tau_{22}(y')] 
\partial_{y_2} \partial_{y_2'}k(y,
                 y') \\
               &  \quad - 
(\tau_{11}(y) y'_{1} + \tau_{12}(y) y'_{2}  ) \partial_{y_1}k(y,
                 y')  - 
(\tau_{12}(y) y'_{1} + \tau_{22}(y) y'_{2}  ) 
\partial_{y_2}k(y,  y') \\
               &  \quad -
(y_{1}\tau_{11}(y')  + y_{2} \tau_{12}(y') )   \partial_{y_1'}k(y,
                 y')  
- (y_{1}\tau_{12}(y')  + y_{2} \tau_{22}(y') ) 
\partial_{y_2'}k(y,   y') \\
               & \quad + (y_1y_1'+y_2y_2')k(y, y')
\end{align*}

As in the previous example, we present simulation results on a rather
modest simulation study. We compare $X \sim p$ a bivariate (centered
scaled) Student with $\ell = 5 $ degrees of freedom with $Y \sim q$ a
bivariate (centered scaled) Student with $\ell$ degrees of freedom,
via the kernelized discrepancies based on the RBF kernel
$k(x, y) = e^{-(x-y)^2/2}$ with $n_1 = n_2 = 100$. The quantiles were
estimated by simulation, with $J=10^{3}$ experiments; we obtained
\begin{center}
  \begin{tabular}{cc}
   2.5\% &       97.5\%  \\
\hline
-0.07256331 &  0.08441458 
  \end{tabular}
\end{center}
(which indicates some asymmetry in the sample distribution).  The
results for $10^3$ simulations run for each value of degrees of
freedom $\ell \in \left\{ 0.1, 1, 5, 10, 100, 1000 \right\}$ (with
$\ell = 5$ corresponding to the null hypothesis) are reported below:
\begin{center}
{\small
\begin{tabular}{c|ccc|c|cccc}
  $\ell$ & 0.1& 1&4& 5& 6& 10& 100& 1000 \\
  \hline 
  kernel  &   0.339  &0.690 &0.056 &0.043 &0.046 &0.038& 0.048 &0.052
  \\
  \hline
\end{tabular}
}
\end{center}
Our naive implementation of the bivariate test appears to have
difficulties in distinguishing the bivariate Student from the
bivariate Gaussian (obtained at $\ell = 1000$). Such an observation is
perhaps not so surprising, see e.g.\ \cite{mcassey2013empirical} where
a similar problem is tested (by different means) with low power for
the case of Gaussian vs Student, see page 1126. The problem of
devising tractable powerful goodness-of-fit tests for multivariate
distributions seems to be difficult; we will concentrate on this in
future publications.

\end{example}

\bigskip
\noindent {\bf{Acknowledgements.}} GM and YS gratefully acknowledges
support by the Fonds de la Recherche Scientifique - FNRS under Grant
MIS F.4539.16.  GR acknowledges partial support from EPSRC grant
EP/K032402/1 and the Alan Turing Institute.  We also thank Christophe
Ley and Guillaume Poly for interesting discussions, as well as Lester Mackey and Steven
Vanduffel for suggesting some references which we had overlooked.

\bibliographystyle{plainnat}

\end{document}